\documentclass[12pt]{amsart}
\usepackage{epsfig} 
\usepackage{color}
\usepackage{psfrag}
\usepackage{graphicx}
\usepackage{amssymb}
\usepackage{amsmath}
\usepackage{latexsym}
\usepackage{enumerate}
\usepackage[mathscr]{eucal}
\usepackage[isolatin]{inputenc}
\usepackage{upref}

\makeatletter
\@namedef{subjclassname@2010}{%
 \textup{2010} Mathematics Subject Classification}
\makeatother

\newtheorem{thm}{Theorem}[section]

\newtheorem{cor}[thm]{Corollary}
\newtheorem{lem}[thm]{Lemma}

\theoremstyle{definition}
\newtheorem{defin}[thm]{Definition}
\newtheorem{rem}[thm]{Remark}

\numberwithin{equation}{section}

 \newcommand{\setN}{\mathbb{N}}

\begin{document}

\baselineskip=17pt

\title[]
{On the downward L\"owenheim-Skolem Theorem for elementary submodels}

\author{Matthias Kunik}
\address{Universit\"{a}t Magdeburg\\
IAN \\
Geb\"{a}ude 02 \\
Universit\"{a}tsplatz 2 \\
D-39106 Magdeburg \\
Germany}
\email{matthias.kunik@ovgu.de}

\date{\today}
\maketitle

\begin{abstract}
We introduce a new definition of a model for a formal mathematical system.
The definition is based upon the substitution in the formal systems,
which allows a purely algebraic approach to model theory.
This is very suitable for applications due to a general syntax used in the formal systems. For our models we present a new proof of the downward L\"owenheim-Skolem Theorem for elementary submodels. 
\end{abstract}



\section{Introduction}\label{intro}
%
%
We present a new definition of a model for the formal mathematical systems
from \cite{Ku}. This definition is based upon the substitution in the formal systems, and it allows a very flexible use for applications
due to the general syntax introduced in \cite{Ku}. 
For these models we obtain a simple new proof of the downward L\"owenheim-Skolem Theorem for elementary submodels using the construction of Henkin-constants. 

In \cite{Ku} we have presented a unified theory for formal systems
inclu\-ding recursive systems closely related to formal grammars, including 
the predi\-cate calculus as well as a formal induction principle.
In \cite[(3.13) Induction Rule (e)]{Ku} we use an own rule of inference for formal induction and special predicate symbols for the underlying recursive systems. The remaining rules of inference \cite[(3.13) (a)-(d)]{Ku} without induction and recursive systems provide an own
Hilbert-style predicate calculus which is similar
to that of Shoenfield in his textbook \cite{Shoenfield}.
The corresponding theory is presented in Section \ref{FMS},
and it is used for the definition of the models
in Section \ref{models}. 
The main result in Section \ref{models} is Lemma \ref{correct}
which states the correctness of the logic calculus, but also
its Corollary \ref{contra_free}. Here the theory of models
is only presented for formal mathematical systems without 
using the Induction Rule from \cite[(3.13)(e)]{Ku}.
In order to see that this is sufficient, we start with a general
formal mathematical system $[M;\mathcal{L}]$, where $M=[S;A;P;B]$
with the underlying recursive system $S=[A_S;P_S;B_S]$. 
Here $A$ is the set of constant- or operation symbols, $P$ the set of pre\-di\-cate symbols and $B$ the set of basis axioms, i.e. the axioms of the given theory,
and we have $A_S \subseteq A$, $P_S \subseteq P$ and $B_S \subseteq B$. 
Then we define
\begin{align*}
N &=[S;A;P;B_S]\,,~ S_{\emptyset}=[\,[\,];[\,];[\,]\,]\,,\\
\Phi &=\{F\,:\,F \mbox{~is~a~provable~closed~formula~of~} [N;\mathcal{L}]\,\}\,,\\
\tilde{M} &= [S_{\emptyset};A;P;B \cup \Phi\,]\,.
\end{align*}
The Induction Rule is only used in $[N;\mathcal{L}]$ in order to generate the axioms in $\Phi$ for $[\tilde{M};\mathcal{L}]$, but can no longer be used in
$[\tilde{M};\mathcal{L}]$. It follows from \cite[(5.1) Proposition]{Ku}
that $[N;\mathcal{L}]$ is free from contradiction.
Note that
$[\tilde{M};\mathcal{L}]$ and $[M;\mathcal{L}]$ both have the same formulas.
Moreover, it is not difficult to obtain from the Deduction Theorem 
\cite[(4.5)]{Ku} that $[\tilde{M};\mathcal{L}]$ and $[M;\mathcal{L}]$
both have the same \textit{provable} formulas.
Hence we will only use formal systems $[M;\mathcal{L}]$
with $M=[S;A;P;B]$ and $S=[\,[\,];[\,];[\,]\,]$, and for these
systems we will shortly write $M=[A;P;B]$.
In Section \ref{LS} we shall formulate the L\"owenheim-Skolem Theorem
\ref{DLST} for our general theory of models
and present a new proof of that theorem. 
In Section \ref{set_theory} we apply our results
to the reduced axiomatic set theory from \cite{Kuset}.

\section{Formal mathematical systems}\label{FMS}
We use notations and results from \cite[Sections 3,4]{Ku}
and from Shoenfield's textbook \cite{Shoenfield}.
In \cite[Section 3]{Ku} a recursive system $S$ closely related 
to Smullyan's elementary formal systems in \cite{Sm} is embedded into a 
so called \textit{formal mathe\-matical system} $M$. 
In \cite[(3.13)]{Ku} we use five rules of inference, namely rules (a)-(e).
Rule (e) enables formal induction with respect to the recursively enumerable relations generated by the underlying recursive system $S$.
For the theory of models we put $S=S_{\emptyset}=[\,[\,];[\,];[\,]\,]$
in order to avoid the use of rule (e). 
Then we can shortly write $M=[A; P; B]$
instead of $M=[S_{\emptyset}; A; P ;B]$
for our formal systems. 
We denote the countably infinite alphabet of variables by
$X=\{\,{\bf x_1}\,,\,{\bf x_2}\,,\,{\bf x_3}\,,\, \ldots \,\}$.

\begin{defin} {\bf Formal mathematical systems}\\ 
Let $M = [A; P; B]$ be a formal mathematical system and ${\mathcal L}$
a given subset of $A$-lists with the properties 
\begin{itemize}
\item[(i)] $X \subseteq {\mathcal L} $\,,
\item[(ii)] $\lambda \frac{\mu}{x} \in {\mathcal L}$ ~
for all $\lambda, \mu \in {\mathcal L}$, $x \in X$\,,
\item[(iii)] all formulas in $B$ contain only argument lists in ${\mathcal L}$\,.
\end{itemize}
Then $[M;{\mathcal L}]$ is also called a formal
mathematical system (with restricted argument lists ${\mathcal L}$).
A formula in $[M;{\mathcal L}]$ is a formula in $M$ 
which has only argument lists in ${\mathcal L}$\,.
A proof $[\Lambda]$ in $[M;{\mathcal L}]$ is a proof in $M$
with the restrictions 
\begin{itemize}
\item[(iv)] the formulas in $[\Lambda]$ and 
the formulas $F$ and $G$ in \cite[(3.13)(a)-(d)]{Ku}
contain only argument lists in ${\mathcal L}$\,,
\item[(v)] there holds $\lambda \in {\mathcal L}$ 
for the list $\lambda$ in \cite[(3.13)(c)]{Ku}\,.
\end{itemize}
Let $\Pi(M;\mathcal{L})$ be the set of formulas provable
in $[M;\mathcal{L}]$ by using only the rules of inference (a),(b),(c),(d)
from \cite[(3.13)]{Ku}.
\end{defin}
\begin{defin}\label{henkin_def}
Let $[M;\mathcal{L}]$ be a formal mathematical system.
\begin{itemize}
\item[(a)] $[M;\mathcal{L}]$ is called \textit{contradictory} if 
$F \in \Pi(M;\mathcal{L})$ for every formula $F$ in $[M;\mathcal{L}]$,
otherwise we say that $[M;\mathcal{L}]$ is \textit{consistent}.
\item[(b)] $[M;\mathcal{L}]$ is called a \textit{Henkin system}
if for every variable $x \in X$ and for every formula $F$ in $[M;\mathcal{L}]$ 
with $\mbox{free}(F) \subseteq \{x\}$ there is a constant $c \in \mathcal{L}$
such that
$\begin{displaystyle}
\to \exists x F \, F\frac{c}{x} \,\in \, \Pi(M;\mathcal{L})\,.
\end{displaystyle}$
We obtain from \cite[(3.19)]{Ku} that the latter condition may be replaced with
\begin{align*}
\leftrightarrow \exists x F \, F\frac{c}{x} \,\in \, \Pi(M;\mathcal{L})\,.
\end{align*}
\item[(c)] $[M;\mathcal{L}]$ is called \textit{complete} if 
$$F \notin \Pi(M;\mathcal{L}) \Leftrightarrow \neg F \in \Pi(M;\mathcal{L}) $$
for every closed formula $F$ in $[M;\mathcal{L}]$\,.
\end{itemize}
\end{defin}
\begin{defin}\label{extensions}
Given are two formal mathematical systems $[M;\mathcal{L}]$ and 
$[M';\mathcal{L}']$ with $M=[A; P; B]$ and $M'=[A'; P'; B']$.
\begin{itemize}
\item[(a)] We say that $[M';\mathcal{L}']$ is an \textit{extension} 
of $[M;\mathcal{L}]$ if
$$
A \subseteq A'\,, \quad 
P \subseteq P'\,, \quad 
\mathcal{L} \subseteq \mathcal{L}' ~\textrm{~and~} ~
\Pi(M;\mathcal{L}) \subseteq \Pi(M';\mathcal{L}') \,.
$$
\item[(b)] Let $[M';\mathcal{L}']$ be an extension of 
$[M;\mathcal{L}]$. If we have in addition
$$
F \in \Pi(M';\mathcal{L}') \implies F \in \Pi(M;\mathcal{L})
$$
for all formulas $F$ in $[M;\mathcal{L}]$, then 
$[M';\mathcal{L}']$ is called a \textit{conservative extension} of 
$[M;\mathcal{L}]$. 
\end{itemize}
\end{defin}

\section{Models of formal mathematical systems}\label{models}
Let $[M;\mathcal{L}]$ be a formal mathematical system  with $M=[A;P;B]$.\\
\textit{A model} $\mathcal{D}$ of $[M;\mathcal{L}]$  
consists of the following ingredients:
\begin{itemize}
\item[(1)] We have a nonempty set $\mathcal{D}_*$, also called the universe of the model. The members $d \in \mathcal{D}_*$ are called the individuals of the universe.
\item[(2)] For each individual $d \in \mathcal{D}_*$ 
we have exactly one name $\alpha_d$ and the set of names
$
\mathcal{N}=\{\alpha_d\,:\,d \in \mathcal{D}_*\,\}\,.
$
It is understood that different individuals have different names and that the names in $\mathcal{N}$ are different from the symbols in $[M;\mathcal{L}]$.
\item[(3)] Put $\hat{A} = A \cup \mathcal{N}$,
$\hat{M}=[\hat{A}; P; B]$ and
\begin{align*} \hat{{\mathcal L}}   = 
\{\, & \lambda \frac{\kappa_1}{x_1}...\frac{\kappa_m}{x_m} \,\,:\,\, 
\lambda \in {\mathcal L},\,
x_1,...,x_m \in X,\,\\ & \kappa_1,...,\kappa_m \in 
\mathcal{N}\,,\, m \geq 0\}\,.
\end{align*}
There results an extended mathematical system 
$[\hat{M};\hat{\mathcal{L}}]$.
Let $\hat{\mathcal{L}}_*$ be the set of all lists $\lambda \in \hat{\mathcal{L}}$
without variables, i.e. we have $\mbox{var}(\lambda)= \emptyset$.
\item[(4)] We have a \textit{surjective} function 
$\mathcal{D} : \hat{\mathcal{L}}_* \to \mathcal{D}_*$ with
\begin{equation*}
\mathcal{D}\left( \lambda \frac{\mu}{x}\right)=
\mathcal{D}\left( \lambda \frac{\alpha_{\mathcal{D}(\mu)}}{x}\right)
\end{equation*}
for all $\mu \in \hat{\mathcal{L}}_*$, for all variables $x \in X$
and for all $\lambda \in \hat{\mathcal{L}}$ with 
$\mbox{var}(\lambda) \subseteq \{ x \}$.
\item[(5)]To each predicate symbol $p \in P$ 
and each $n \in \setN_0=\{0,1,2,3,\ldots\}$ we assign an $n$-ary predicate $p_n \subseteq \mathcal{D}_*^n$. Especially for $n=0$ we either have a truth value $p_0=\top$ or $p_0 = \bot$.
\item[(6)] We have an extension of the function $\mathcal{D}$, which assigns a truth value to each closed formula of $[\hat{M};\hat{\mathcal{L}}]$.
This extension is also denoted by $\mathcal{D}$ and is given by
\begin{itemize}
\item[6.1] $\begin{displaystyle}
\mathcal{D}(\sim \lambda,\mu)=\top \Leftrightarrow
\mathcal{D}(\lambda)=\mathcal{D}(\mu)
\end{displaystyle}$ 
for all $\lambda, \mu \in \hat{\mathcal{L}}_*$,
\item[6.2] $\mathcal{D}(p)=p_0 \in \{\top ,\bot\}$ for all $p \in P$,
$\begin{displaystyle}
\mathcal{D}(p \, \lambda_1,\ldots,\lambda_n)=\top \Leftrightarrow
(\mathcal{D}(\lambda_1),\ldots,\mathcal{D}(\lambda_n)) \in p_n
\end{displaystyle}$ for all $n \in \setN=\{1,2,3,\ldots\}$ and all
$\lambda_1,\ldots,\lambda_n \in \hat{\mathcal{L}}_*$\,.
\item[6.3] We have for all closed formulas $F,G \mbox{~of~} [\hat{M};\hat{\mathcal{L}}]$:\\
$\begin{displaystyle}
\mathcal{D}(\neg F)=\top \Leftrightarrow
\mathcal{D}(F)=\bot\,,
\end{displaystyle}$ \\
$\begin{displaystyle}
\mathcal{D}(\to F G)=\top \Leftrightarrow
(\mathcal{D}(F) \Rightarrow \mathcal{D}(G))\,,
\end{displaystyle}$ \\
$\begin{displaystyle}
\mathcal{D}(\leftrightarrow F G)=\top \Leftrightarrow
(\mathcal{D}(F) \Leftrightarrow \mathcal{D}(G))\,,
\end{displaystyle}$ \\
$\begin{displaystyle}
\mathcal{D}(\& \, F G)=\top \Leftrightarrow
(\mathcal{D}(F) ~\mbox{and}~ \mathcal{D}(G))\,,
\end{displaystyle}$ \\
$\begin{displaystyle}
\mathcal{D}(\vee \, F G)=\top \Leftrightarrow
(\mathcal{D}(F) ~\mbox{or}~ \mathcal{D}(G))\,.
\end{displaystyle}$ 
\item[6.4] We have for all $x \in X$ and all formulas 
$F \mbox{~of~}  [\hat{M};\hat{\mathcal{L}}]$ with $\mbox{free}(F) \subseteq \{x\}$:\\
$\begin{displaystyle}
\mathcal{D}(\forall x \, F)=\top \Leftrightarrow
\left( \mathcal{D}\left(F \frac{\lambda}{x}\right)=\top ~\mbox{for all}~
\lambda \in \hat{\mathcal{L}}_* \right) \,,
\end{displaystyle}$ \\
$\begin{displaystyle}
\mathcal{D}(\exists x \, F)=\top 
\Leftrightarrow
\left(~\mbox{there exists}~\lambda \in \hat{\mathcal{L}}_* 
~\mbox{with}~
\mathcal{D}\left(F \frac{\lambda}{x}\right)=\top
\right)\,.
\end{displaystyle}$ 
\end{itemize}
\item[(7)] Let $F$ be a formula in 
$[\hat{M};\hat{\mathcal{L}}]$ with 
$\mbox{free}(F) = \{x_1,\ldots,x_m\}$, $x_1,\ldots,x_m \in X$ 
and $m \geq 0$. We say that $F$ is valid in $\mathcal{D}$ iff
$\begin{displaystyle}
\mathcal{D}\left(
F \frac{\lambda_1}{x_1}\ldots\frac{\lambda_m}{x_m}
\right)=\top
\end{displaystyle}$
for all $\lambda_1,\ldots,\lambda_m \in \hat{\mathcal{L}}_* $.
Note that this simply means 
$\begin{displaystyle}
\mathcal{D}\left(
\forall x_1 \ldots \forall x_m F 
\right)=\top
\end{displaystyle}$\,.
\item[(8)] We require that every formula $F \in B$ is valid in 
$\mathcal{D}$. Then we say that $\mathcal{D}$ is a model
for $[M;\mathcal{L}]$.
\end{itemize}
\begin{rem} 
(a) We say that an extended mapping $\mathcal{D}$ which satisfies
only Conditions (1)-(7) is a \textit{structure} for the formal mathematical system 
$[M;\mathcal{L}]$. Finally, Condition (8) makes it a \textit{model} for
$[M;\mathcal{L}]$. 

(b) We use a Hilbert-style calculus for our formal mathematical systems. Hence we can use the substitution rule (c) and formulas with free variables in our axioms.\\
Let $F$ be a formula in $[\hat{\mbox{M}};\hat{\mathcal{L}}]$ and $x \in X$.
Then $F \in \Pi(\hat{\mbox{M}};\hat{\mathcal{L}})$ iff
$\forall x\, F \in \Pi(\hat{\mbox{M}};\hat{\mathcal{L}})$
from \cite[(3.11)(a),(3.13)(b)(d)]{Ku}. This matches well with Condition (7) in our definition of the models.
\end{rem}

\begin{defin}\label{degree}{\bf Degree of a formula}\\
For all formulas $F$, $G$ of a formal mathematical system $[M;\mathcal{L}]$ 
we define their degree $\mbox{deg}(F)$, $\mbox{deg}(G)$ as follows: 
\begin{itemize}
\item If $F$ is a prime formula then $\mbox{deg}(F)=0$\,,
\item $\mbox{deg}(\neg F)=\mbox{deg}(F)+1$\,,
\item $\mbox{deg}(J F G)=\max(\mbox{deg}(F),\mbox{deg}(G))+1$
for $J \in \{\to,\vee,\&,\leftrightarrow\}$\,,
\item $\mbox{deg}(\forall x F)=\mbox{deg}(\exists x F)=\mbox{deg}(F)+1$
for any $x \in X$\,.
\end{itemize}
\end{defin}

\begin{lem}\label{structure_lemma}
Let $\mathcal{D}$ be a structure for $[M;\mathcal{L}]$. Then we have:
\begin{itemize}
\item[(a)] $d = \mathcal{D}\left( \alpha_d \right)$ for all $d \in \mathcal{D}_*$.
\item[(b)]For $x \in X$ and every formula $H$ in $[\hat{M};\hat{\mathcal{L}}]$
with $\mbox{free}(H) \subseteq \{x\}$ and for all $\mu \in \hat{\mathcal{L}}_*$
we have
\begin{equation*}
\mathcal{D}\left( H \frac{\mu}{x} \right) = 
\mathcal{D}\left( H \frac{\alpha_{\mathcal{D}(\mu)}}{x} \right)\,.
\end{equation*}
\end{itemize}
\end{lem}
\begin{proof}
(a) We obtain from Condition (4) with $\lambda = x \in X$ that
\begin{equation*}
\mathcal{D}(\mu)=\mathcal{D}\left( x \frac{\mu}{x} \right)
=\mathcal{D}\left( x \frac{\alpha_{\mathcal{D}(\mu)}}{x} \right)
=\mathcal{D}\left( \alpha_{\mathcal{D}(\mu)} \right)
\end{equation*}
for all $\mu \in \hat{\mathcal{L}}_*$.
Note that $\mathcal{D}: \hat{\mathcal{L}}_* \to \mathcal{D}_*$ is surjective.
Hence for each $d \in \mathcal{D}_*$ we have $\mu \in \hat{\mathcal{L}}_*$
with
\begin{equation*}
d=\mathcal{D}(\mu)
=\mathcal{D}\left( \alpha_{\mathcal{D}(\mu)} \right)
=\mathcal{D}(\alpha_{d})\,.
\end{equation*}

(b) We employ induction with respect to the degree of the formulas:
For all $\lambda, \nu \in \hat{\mathcal{L}}$ with 
$\mbox{var}(\lambda) \subseteq \{x\}$, 
$\mbox{var}(\nu) \subseteq \{x\}$ and for all 
$\mu \in \hat{\mathcal{L}}_*$ we have
\begin{align*}
\mathcal{D}\left( \lambda \frac{\mu}{x} \right)=
\mathcal{D}\left( \lambda \frac{\alpha_{\mathcal{D}(\mu)}}{x} \right)
\,,\quad
\mathcal{D}\left( \nu \frac{\mu}{x} \right)=
\mathcal{D}\left( \nu \frac{\alpha_{\mathcal{D}(\mu)}}{x} \right)
\,,
\end{align*}
and hence
\begin{align*}
& \mathcal{D}\left( \mbox{SbF}(\sim \lambda,\nu; \mu; x) \right)=
\mathcal{D}\left( \sim \lambda\frac{\mu}{x},\, \nu \frac{\mu}{x} \right)=\top
\quad \Leftrightarrow\\
& 
\mathcal{D}\left( \lambda\frac{\mu}{x}\right) = 
\mathcal{D}\left(  \nu \frac{\mu}{x} \right)
\quad \Leftrightarrow \quad
\mathcal{D}\left( \lambda \frac{\alpha_{\mathcal{D}(\mu)}}{x} \right) = 
\mathcal{D}\left( \nu \frac{\alpha_{\mathcal{D}(\mu)}}{x} \right)
\quad \Leftrightarrow\\
& 
\mathcal{D}\left( \sim \lambda\frac{\alpha_{\mathcal{D}(\mu)}}{x},\, \nu \frac{\alpha_{\mathcal{D}(\mu)}}{x} \right)=
\mathcal{D}\left( \mbox{SbF}(\sim \lambda,\nu; \, \alpha_{\mathcal{D}(\mu)}; x) 
\right)= \top\,.
\end{align*}

We obtain for $p \in P$ that
\begin{align*}
\mathcal{D}\left( \mbox{SbF}(p; \mu; x) \right)=
\mathcal{D}\left(p \right)=
\mathcal{D}\left( \mbox{SbF}(p; \alpha_{\mathcal{D}(\mu)}; x) \right)
\in \{\top, \bot \}\,,
\end{align*}
and if $\lambda_1,\ldots,\lambda_n \in \hat{\mathcal{L}}$
with $\mbox{var}(\lambda_j) \subseteq \{x\}$ for all $j=1,\ldots,n$:
\begin{align*}
& \mathcal{D}\left( \mbox{SbF}(p \, \lambda_1,\ldots, \lambda_n \,; \mu; x) \right)=
\mathcal{D}\left( p \, \lambda_1\frac{\mu}{x},\ldots, \lambda_n 
\frac{\mu}{x} \,\right)=\top
\quad \Leftrightarrow\\
& 
\left(\mathcal{D}\left( \lambda_1\frac{\mu}{x}\right) \,,\ldots, 
\mathcal{D}\left(  \lambda_n \frac{\mu}{x} \right)\right) \in p_n
\quad\Leftrightarrow \\
& \left(\mathcal{D}\left( \lambda_1\frac{\alpha_{\mathcal{D}(\mu)}}{x}\right) \,,\ldots, 
\mathcal{D}\left(  \lambda_n \frac{\alpha_{\mathcal{D}(\mu)}}{x} \right)\right) 
\in p_n \quad \Leftrightarrow \\
& 
\mathcal{D}\left( \mbox{SbF}(p \, \lambda_1\,,\ldots ,\lambda_n \,; 
\alpha_{\mathcal{D}(\mu)}; x) \right)= \top\,.
\end{align*}
We have proven the statement from Lemma \ref{structure_lemma}(b)
for all formulas $H$ of degree zero, i.e. for all prime formulas.

Suppose that for some $n \in \setN$
we have already proven the statement of Lemma \ref{structure_lemma}(b)
for all formulas $H$ in $[\hat{M};\hat{\mathcal{L}}]$
with $\mbox{deg}(H)<n$ and $\mbox{free}(H) \subseteq \{x\}$\,.
We define the propositional functions
\begin{align*}
\Pi_{\neg}\,:\, \{\top,\bot \} \to \{\top,\bot \} \mbox{\,~and~\,}
\Pi_{\to}\,, \Pi_{\leftrightarrow}\,, \Pi_{\&}\,, \Pi_{\vee}\,:\,
\{\top,\bot \}^2 \to \{\top,\bot \} 
\end{align*}
with
\begin{align*}
& \Pi_{\neg}(\xi)=(\mbox{not}\, \xi)\,,\\
& \Pi_{\to}(\xi_1,\xi_2)=(\xi_1 \Rightarrow \xi_2)\,, \quad
  \Pi_{\leftrightarrow}(\xi_1,\xi_2)=(\xi_1 \Leftrightarrow \xi_2)\,, \\
& \Pi_{\&}(\xi_1,\xi_2)=(\xi_1 ~ \mbox{and}~\xi_2)\,, \quad
  \Pi_{\vee}(\xi_1,\xi_2)=(\xi_1 ~ \mbox{or}~\xi_2)\,.
\end{align*}
We suppose that $F$, $G$ are formulas in $[\hat{M};\hat{\mathcal{L}}]$
with $\mbox{free}(F) \subseteq \{x\}$, $\mbox{free}(G) \subseteq \{x\}$
and $\mbox{deg}(F) < n$, $\mbox{deg}(G) < n$\,.
Then we obtain for all $\mu \in \hat{\mathcal{L}}_*$:
\begin{align*}
&\mathcal{D}\left( \mbox{SbF}(\neg F; \mu; x) \right)=
\mathcal{D}\left(\neg \mbox{SbF}(F; \mu; x) \right)=\\
&\Pi_{\neg}\left(\mathcal{D}\left( \mbox{SbF}(F; \mu; x) \right)\right)=
\Pi_{\neg}\left(\mathcal{D}\left( \mbox{SbF}(F; \alpha_{\mathcal{D}(\mu)}; x) \right)\right)=\\
&\mathcal{D}\left(\neg \mbox{SbF}(F; \alpha_{\mathcal{D}(\mu)}; x) \right)=
\mathcal{D}\left(\mbox{SbF}(\neg F; \alpha_{\mathcal{D}(\mu)}; x) \right)
\in \{\top, \bot \}\,,
\end{align*}
and for $J \in \{\to;\,\leftrightarrow;\,\&;\,\vee \}$:
\begin{align*}
&\mathcal{D}\left( \mbox{SbF}(J\, F \, G; \mu; x) \right)=
\mathcal{D}\left(J \, \mbox{SbF}(F; \mu; x) \, \mbox{SbF}(G; \mu; x)\right)=\\
&\Pi_{J}\left(\mathcal{D}\left( \mbox{SbF}(F; \mu; x) \right),
\mathcal{D}\left( \mbox{SbF}(G; \mu; x) \right)\right)=\\
&\Pi_{J}\left(\mathcal{D}\left( \mbox{SbF}(F; \alpha_{\mathcal{D}(\mu)}; x),
\mbox{SbF}(G; \alpha_{\mathcal{D}(\mu)}; x)
 \right)\right)=\\
&\mathcal{D}\left(J \, \mbox{SbF}(F; \alpha_{\mathcal{D}(\mu)}; x)\,
\mbox{SbF}(G; \alpha_{\mathcal{D}(\mu)}; x) \right)=
\mathcal{D}\left(\mbox{SbF}(J\, F\, G; \alpha_{\mathcal{D}(\mu)}; x) \right)
\,.
\end{align*}
We finally suppose that $H$ is a formula in $[\hat{M};\hat{\mathcal{L}}]$
with $\mbox{deg}(H) < n$, $x, y \in X$ and $\mbox{free}(Q\, y\, H) \subseteq \{x\}$,
$Q \in \{\forall; \exists \}$ and $\mu \in \hat{\mathcal{L}}_*$.

\textit{Case 1:} $\mbox{free}(Q\, y\, H) = \emptyset$\,. Then 
\begin{align*}
\mbox{SbF}(Q\, y\, H; \mu; x) = Q\, y\, H =
\mbox{SbF}(Q\, y\, H; \alpha_{\mathcal{D}(\mu)}; x)\,.
\end{align*}

\textit{Case 2:} $\mbox{free}(Q\, y\, H) = \{x\}$ \,.
Then $x \neq y$, and we have
\begin{align*}
& \mathcal{D}\left(\mbox{SbF}(\forall \, y\, H; \mu; x) \right) = 
\mathcal{D}\left(\forall \, y\, H \frac{\mu}{x} \right) = \top \quad
\Leftrightarrow\\
& \mathcal{D}\left(H \frac{\mu}{x} \frac{\lambda}{y} \right) = \top \quad
\mbox{for~all~} \lambda \in \hat{\mathcal{L}}_* \quad
\Leftrightarrow\\
& \mathcal{D}\left(H \frac{\lambda}{y} \frac{\mu}{x} \right) = \top \quad
\mbox{for~all~} \lambda \in \hat{\mathcal{L}}_* \quad
\Leftrightarrow
\end{align*}
\begin{align*}
& \mathcal{D}\left(H \frac{\lambda}{y} 
\frac{\alpha_{\mathcal{D}(\mu)}}{x} \right) = \top \quad
\mbox{for~all~} \lambda \in \hat{\mathcal{L}}_* \quad
\Leftrightarrow\\
& \mathcal{D}\left(H 
\frac{\alpha_{\mathcal{D}(\mu)}}{x} \frac{\lambda}{y} \right) = \top \quad
\mbox{for~all~} \lambda \in \hat{\mathcal{L}}_* \quad
\Leftrightarrow\\
& \mathcal{D}\left(\forall y \, H 
\frac{\alpha_{\mathcal{D}(\mu)}}{x} \right) = \top \quad
\Leftrightarrow\\
& \mathcal{D}\left(\mbox{SbF} \left( \forall y \, H \, ;\,
\alpha_{\mathcal{D}(\mu)}; x \right) \right)= \top \,.
\end{align*}
We proceed in the same way in order to prove that 
\begin{align*}
\mathcal{D}\left(\mbox{SbF} \left( \exists y \, H \, ;\,
\mu; x \right) \right) = 
\mathcal{D}\left(\mbox{SbF} \left( \exists y \, H \, ;\,
\alpha_{\mathcal{D}(\mu)}; x \right) \right) \,.
\end{align*}
\end{proof}
\begin{rem}\label{shoenfield_structure}
Here we consider a special kind of structures for formal mathematical systems,
but one that is mainly considered in model theory:

To each constant or function symbol $a \in A$ we assign 
a fixed arity $n \in \setN_0$. 
For $n=0$ we say that $a$ is a constant symbol, and for $n \geq 1$ we say that $a$ is an $n$-ary function symbol.
Then $\mathcal{L}$ consists only on terms which are generated by the following rules.
\begin{itemize}
\item[1.] We have $x \in \mathcal{L}$ 
for all variables $x \in X=\{\,{\bf x_1}\,,\,{\bf x_2}\,,\,{\bf x_3}\,,\, \ldots \,\}$.
\item[2.] We have $a \in \mathcal{L}$ for all constant symbols $a \in A$.
\item[3.] Let $n>0$ and let $f$ be an $n$-ary function symbol in $A$.\\
Then $f(\lambda_1 \ldots \lambda_n) \in \mathcal{L}$ for all terms
$\lambda_1,\ldots,\lambda_n \in \mathcal{L}$.
\end{itemize}
To obtain a structure $\mathcal{D}$ 
for $[M;\mathcal{L}]$ we are following the textbook 
of Shoenfield \cite[Section 2.5]{Shoenfield}.

Given is a nonempty set $\mathcal{D}_*$. Its members are called individuals.
By induction with respect to the terms we shall now define an individual
$\mathcal{D}(t)$ for each variable-free term $t$ as follows:
\begin{itemize}
\item[i)]  Each individual $d \in \mathcal{D}_*$ has its own name $\alpha_d$.
Let $\mathcal{N}$ be the set of all these names.
We suppose that the symbols in $[M;\mathcal{L}]$ are different
from the symbols in $\mathcal{N}$ and that 
$d_1=d_2 \Leftrightarrow \alpha_{d_1}=\alpha_{d_2}$
for all $d_1,d_2 \in \mathcal{D}_*$.
Define $[\hat{M};\hat{\mathcal{L}}]$ and $\hat{\mathcal{L}}_*$
as in Condition (3).
\item[ii)] To each constant symbol $a \in \hat{\mathcal{L}}_*$ we assign a value $\mathcal{D}(a) \in \mathcal{D}_*$ 
with $d=\mathcal{D}(\alpha_{d})$ for all $d \in \mathcal{D}_*$.
To each $n$-ary function symbol $f$ we assign an $n$-ary function
$f_{\mathcal{D}} : \mathcal{D}_*^n \to \mathcal{D}_*$.
Let $t_1,\ldots,t_n \in \hat{\mathcal{L}}_*$ be variable-free terms and let
$\mathcal{D}(t_1),\ldots,\mathcal{D}(t_n) \in \mathcal{D}_*$
be defined previously. Then we put
$\begin{displaystyle}
\mathcal{D}(f(t_1\ldots t_n))=
f_{\mathcal{D}}(\mathcal{D}(t_1),\ldots,\mathcal{D}(t_n))
\end{displaystyle}$\,.
\end{itemize}
The remai\-ning steps of the construction are due to Conditions (5)-(7).
Then $\mathcal{D}$ is a structure for $[M;\mathcal{L}]$, and
Shoenfield's first lemma in \cite[Section 2.5]{Shoenfield}
is just a variant of our Lemma \ref{structure_lemma}.

\end{rem}
\begin{lem}\label{correct}
Let $\mathcal{D}$ be a model for $[M;\mathcal{L}]$.  Using the notation from Condition (3) we obtain for every formula $F$ in $[\hat{M};\hat{\mathcal{L}}]$:
\begin{equation*}
F \in \Pi(\hat{M};\hat{\mathcal{L}}) \Rightarrow F \mbox{~is~valid~in~}
\mathcal{D}\,.
\end{equation*}
Here $\Pi(\hat{M};\hat{\mathcal{L}})$ denotes the set of provable formulas 
in $[\hat{M};\hat{\mathcal{L}}]$.
\end{lem}
\begin{proof}
For any proof $[\Lambda]$ in $[\hat{M};\hat{\mathcal{L}}]$
and any step $F$ of $[\Lambda]$ we have to show that $F$ is valid in $\mathcal{D}$.
This is true for $[\Lambda]=[\,]$. Let $[\Lambda]=[F_1;\ldots;F_l]$ be a proof
in $[\hat{M};\hat{\mathcal{L}}]$ with the steps $F_1,\ldots,F_l$.
We assume that $F_1,\ldots,F_l$ are valid in $\mathcal{D}$ and use induction
with respect to the rules of inference (a)-(d).
We put $\mbox{gen}(G)=\forall x_1 \ldots \forall x_j G$
for every formula $G$ in $[\hat{M};\hat{\mathcal{L}}]$ 
with $\mbox{free}(G)=\{x_1,\ldots,x_j\}$ and 
distinct variables $x_1,\ldots,x_j \in X$,
ordered according to their first occurrence in $G$.

\textit{Rule (a)}. We have to show that every axiom $F$ of $[\hat{M};\hat{\mathcal{L}}]$
is valid in $\mathcal{D}$. Then the extended proof 
$[\Lambda_*]=[\Lambda;F]$ will also satisfy the required property.

\textit{Basis axioms.}  The formal systems $[M;\mathcal{L}]$ and $[\hat{M};\hat{\mathcal{L}}]$
both have the same basis axioms. Since $\mathcal{D}$
is a model of $[M;\mathcal{L}]$, we conclude that every basis axiom $F$
of $[\hat{M};\hat{\mathcal{L}}]$ is valid in $\mathcal{D}$.

\textit{Axioms of the propositional calculus.}  
Let $\alpha=\alpha(\xi_1,\ldots,\xi_j)$ be a gene\-rally valid
propositional function of the propositional variables $\xi_1,\ldots,\xi_j$
according to \cite[(3.8)]{Ku}. Let $F_1, \ldots, F_j$ be formulas in
$[\hat{M};\hat{\mathcal{L}}]$. Then $F = \alpha(F_1,\ldots,F_j)$ is
an axiom of the propositional calculus in $[\hat{M};\hat{\mathcal{L}}]$.
Assume that $\mbox{free}(F_k) \subseteq \{x_1,\ldots,x_m\}$ for all $k=1,\ldots,j$.
We use the definition of the function $\overline{\psi}$
from \cite[(3.8)]{Ku}, choose $\lambda_1,\ldots,\lambda_m \in \hat{\mathcal{L}}_*$ and obtain for 
$\begin{displaystyle} \overline{\psi}(\xi_k)=\mathcal{D}
\left(
F_k \frac{\lambda_1}{x_1}\ldots \frac{\lambda_m}{x_m}
\right)
\end{displaystyle}$:
\begin{align*}
& F \frac{\lambda_1}{x_1} \ldots \frac{\lambda_m}{x_m} =\alpha
\left(
F_1 \frac{\lambda_1}{x_1}\ldots \frac{\lambda_m}{x_m}\,,
\ldots\,,
F_j \frac{\lambda_1}{x_1}\ldots \frac{\lambda_m}{x_m}
\right) \Rightarrow\\
& \mathcal{D}\left( F \frac{\lambda_1}{x_1}\ldots\frac{\lambda_m}{x_m} \right) 
= \overline{\psi}(\alpha)=\top\,.
\end{align*}
We see that $F$ is valid in $\mathcal{D}$.

\textit{Axioms of equality.} 
Let $F$ be an axiom of equality from \cite[(3.10)]{Ku} in
$[\hat{M};\hat{\mathcal{L}}]$. Due to
Lemma \ref{structure_lemma} it is sufficient for the evaluation of $\mathcal{D}$
if we replace all free variables in $F$ by names.\\
(a) For $F \, = \sim x,x$ and $d \in \mathcal{D}_*$ we have
$\mathcal{D}(\alpha_d)=d$, 
$\mathcal{D}(\sim \alpha_d,\alpha_d)=\top$\,.\\
(b) For $F = \, \to \mbox{SbF}(\sim \lambda, \mu ; x; y) \to \, \sim x,y \sim \lambda,\mu$ with $x,y \in X$ and $\lambda, \mu \in \hat{\mathcal{L}}$ we may assume that
$x \neq y$. Otherwise $F$ is an axiom of the propositional calculus.
After having replaced all free variables other than $x$ and $y$ by names we may also assume that $\mbox{free}(\lambda) \subseteq \{x,y\}$ and
$\mbox{free}(\mu) \subseteq \{x,y\}$. For $d, d' \in \mathcal{D}_*$ 
we obtain
\begin{align*}
& \mathcal{D}\left(F \frac{\alpha_d}{x} \frac{\alpha_{d'}}{y}\right)=\\
& \mathcal{D}\left(
\, \to \, \sim \lambda \frac{x}{y} \frac{\alpha_d}{x}, 
\mu \frac{x}{y} \frac{\alpha_d}{x} \to \, \sim \alpha_{d} ,\alpha_{d'} 
\sim \lambda \frac{\alpha_d}{x} \frac{\alpha_{d'}}{y},
\mu \frac{\alpha_d}{x} \frac{\alpha_{d'}}{y}
\right)\,.
\end{align*}
For $d \neq d'$ we have 
$\mathcal{D}\left(\sim \alpha_{d} ,\alpha_{d'} \right) = \bot$ and 
$\mathcal{D}\left(F \frac{\alpha_d}{x} \frac{\alpha_{d'}}{y}\right)=\top$.\\
Otherwise we have $d = d'$ and 
\begin{align*}
\sim \lambda \frac{x}{y} \frac{\alpha_d}{x}, 
\mu \frac{x}{y} \frac{\alpha_d}{x} = \,
\sim \lambda \frac{\alpha_d}{x} \frac{\alpha_d}{y},
\mu \frac{\alpha_d}{x} \frac{\alpha_d}{y} =\,
\sim \lambda \frac{\alpha_{d}}{x} \frac{\alpha_{d'}}{y},
\mu \frac{\alpha_{d}}{x} \frac{\alpha_{d'}}{y}
\,.
\end{align*}
We see again 
$\mathcal{D}\left(F \frac{\alpha_d}{x} \frac{\alpha_{d'}}{y}\right)=\top$.\\
(c) Here we consider the axiom of equality
$$F = \,\to \, \sim x_1,y_1 \ldots \to \,\sim x_n,y_n \, \to
p\,x_1,\ldots,x_n \,p\,y_1,\ldots,y_n$$ 
with $p \in P$, $n \geq 1$ and
$x_1,\ldots,x_n , y_1,\ldots,y_n \in X$.
Let $F_*$ result from $F$ if we replace all (free) variables in $F$ with names.
Then $\mathcal{D}\left(F_*\right)=\top$ for
 \begin{align*}
F_* = \,\to \, \sim \lambda_1,\lambda'_1 \ldots 
\to \,\sim \lambda_n,\lambda'_n \, \to
p\,\lambda_1,\ldots,\lambda_n \,p\,\lambda'_1,\ldots,\lambda'_n
\end{align*}
whenever $\lambda_j \neq \lambda'_j$ 
for two names $\lambda_j , \lambda'_j$\,.
Otherwise $\lambda_j = \lambda'_j$ for all $j=1,\ldots,n$, and 
also in this case $\mathcal{D}\left(F_*\right)=\top$.

\textit{Quantifier axioms.} 
(a) Let
$\begin{displaystyle}
\mbox{free}(\forall x F ) = \{x_1,\ldots,x_m\}
\end{displaystyle}$\,.
We put\\ $H = \, \rightarrow \forall x F ~ F$ and 
note that $x \notin \{x_1,\ldots,x_m\}$. We see
$\top = \mathcal{D}(\mbox{gen}(H))$ iff
$$\top = \mathcal{D}(\mbox{SbF}(\tilde{H};\mu;x)) = 
\mathcal{D}(\rightarrow \forall x \tilde{F} ~ 
\mbox{SbF}(\tilde{F};\mu;x))$$
for all $\mu, \lambda_1, \ldots \lambda_m \in \hat{{\mathcal L}}_*$,
using $\tilde{F} = F\frac{\lambda_1}{x_1}\ldots\frac{\lambda_m}{x_m}$ and
$$
\tilde{H} = H\frac{\lambda_1}{x_1}\ldots\frac{\lambda_m}{x_m}
= \, \rightarrow \forall x \tilde{F} ~ \tilde{F}
$$
as abbreviations. 

Now $\top = \mathcal{D}(\forall x \tilde{F})$
implies indeed $\top = \mathcal{D}(\mbox{SbF}(\tilde{F};\mu;x))$ for all\\ 
$\mu, \lambda_1, \ldots \lambda_m \in \hat{{\mathcal L}}_*$,
independent of $x \in \mbox{free}(F)$ or $x \notin \mbox{free}(F)$.

(b) Suppose that $x \in X$, that $F$, $G$ are formulas
in $[\hat{M};\hat{\mathcal{L}}]$
and that $x \notin \mbox{free}(F)$, 
$\mbox{free}(\forall x \rightarrow F G)=
\mbox{free}(\rightarrow F \, \forall x G)=\{x_1;\ldots;x_m\}$. We put 
$$H = \, \rightarrow \forall x \rightarrow F G ~ \rightarrow F \, \forall x G\,,$$ fix arbitrary lists
$\lambda_1, \ldots \lambda_m \in \hat{{\mathcal L}}_*$ and make use of the abbreviations
$\tilde{F} = F\frac{\lambda_1}{x_1}\ldots\frac{\lambda_m}{x_m}$ and
$\tilde{G} = G\frac{\lambda_1}{x_1}\ldots\frac{\lambda_m}{x_m}$.
We have
$$
\tilde{H} = H\frac{\lambda_1}{x_1}\ldots\frac{\lambda_m}{x_m}
= \, \, \rightarrow \forall x \rightarrow \tilde{F} \tilde{G} ~ 
\rightarrow \tilde{F} \, \forall x \tilde{G}
$$
with $\mbox{free}(\tilde{H})=\emptyset$\,.
In order to show $\top = \mathcal{D}(\tilde{H})$ we assume\\ 
$\top = \mathcal{D}( \forall x \rightarrow \tilde{F} \tilde{G})$
and note that $x \notin \mbox{free}(\tilde{F})$.
Then
$$\qquad \top = \mathcal{D}( \forall x \rightarrow \tilde{F} \tilde{G})
~\mbox{iff}~
\top = \mathcal{D}( \rightarrow \tilde{F} \, \mbox{SbF}(\tilde{G};\lambda;x))$$
for all $\lambda \in \hat{{\mathcal L}}_*$.
Hence $\top = \mathcal{D}(\tilde{F})$
implies $\top = \mathcal{D}(\mbox{SbF}(\tilde{G};\lambda;x))$ for all 
$\lambda \in \hat{\mathcal{L}}_*$, i.e. $\top = \mathcal{D}(\tilde{F})$
implies $\top = \mathcal{D}(\, \forall x \tilde{G})$,
and we have shown $\top = \mathcal{D}(\,\to \tilde{F}\, \forall x \tilde{G})$
and $\top = \mathcal{D}(\tilde{H})$.

(c) Let the formula $H = \, \leftrightarrow \neg \forall x \neg F \, \exists x F$
in $[\hat{M};\hat{\mathcal{L}}]$ be a quantifier axiom 
according to \cite[(3.11)(c)]{Ku} with $\mbox{free}(H)=\{x_1,\ldots,x_m\}$.
We note that $x \notin \mbox{free}(H)$ and that
$\mbox{free}(F) \subseteq \{x,x_1,\ldots,x_m\}$.
We prescribe $\lambda_1,\ldots,\lambda_m \in \hat{\mathcal{L}}_*$
and put
$$
\tilde{F} = F \frac{\lambda_1}{x_1}\ldots \frac{\lambda_m}{x_m}\,,~
\tilde{H} = \, \leftrightarrow \neg \forall x \neg \tilde{F} \, \exists x \tilde{F}\,.
$$
Then we have 
$\begin{displaystyle} 
\tilde{H} = H \frac{\lambda_1}{x_1}\ldots \frac{\lambda_m}{x_m}
\end{displaystyle}$ and
\begin{align*}
& \mathcal{D}(\exists x \tilde{F})=\top \quad \Leftrightarrow \quad
\left( \mbox{there~exists~}\lambda \in \hat{{\mathcal L}}_* \mbox{~with~}
\mathcal{D}(\tilde{F}\frac{\lambda}{x})=\top \right) \quad \Leftrightarrow\\
& 
\left( \mbox{there~exists~}\lambda \in \hat{{\mathcal L}}_* \mbox{~with~}
\mathcal{D}(\neg \tilde{F}\frac{\lambda}{x})=\bot \right) \quad \Leftrightarrow\\
& 
\left( \mbox{not~for~all~}\lambda \in \hat{{\mathcal L}}_* \mbox{~we~have~}
\mathcal{D}(\neg \tilde{F}\frac{\lambda}{x})=\top \right) \quad \Leftrightarrow\\
& \mathcal{D}(\neg \forall x \neg \tilde{F})=\top \,.
\end{align*}
We conclude that $\mathcal{D}(\tilde{H})=\top$ and see that 
$H$ is valid in $\mathcal{D}$.

\textit{Rule (b).} Suppose that $F$ and $H=\rightarrow F G$ are both steps
of the proof $[\Lambda]=[F_1;\ldots;F_l]$
with $\mbox{free}(\rightarrow F G)=\{x_1,\ldots,x_m\}$\,.

Then $\top = \mathcal{D}(\mbox{gen}(F))$ and 
$\top = \mathcal{D}(\mbox{gen}(H))$ from our induction hypothesis. Fix
$\lambda_1, \ldots \lambda_m \in \hat{{\mathcal L}}_*$ and put
$\tilde{F} = F\frac{\lambda_1}{x_1}\ldots\frac{\lambda_m}{x_m}$,
$\tilde{G} = G\frac{\lambda_1}{x_1}\ldots\frac{\lambda_m}{x_m}$.\\
For 
$$
\tilde{H} = H\frac{\lambda_1}{x_1}\ldots\frac{\lambda_m}{x_m}
= \, \, \rightarrow \tilde{F} \tilde{G} 
$$
we have no free variables in $\tilde{F}, \tilde{G}, \tilde{H}$. 

Now
$\top = \mathcal{D}(\tilde{F})$,
$\top = \mathcal{D}(\tilde{H})$ and $\top = \mathcal{D}(\tilde{G})$.
Note that substitutions of variables in $\{x_1,\ldots,x_m\}$
not occurring in $F$ or $G$ are allowed, because they do not have any effect.
We obtain that the extended proof $[\Lambda_*]=[\Lambda;G]$
also satisfies our statement.

\textit{Rule (c).} Let $x \in X$ and suppose that $F$
is a step of the proof $[\Lambda]=[F_1;\ldots;F_l]$. 
Let $\lambda \in \hat{{\mathcal L}}$ and
suppose that there holds the condition $\mbox{CF}(F;\lambda;x)$. 
Note that $\top = \mathcal{D}(\mbox{gen}(F))$ from our induction hypothesis.

Without loss of generality we may assume that $x \in \mbox{free}(F)$,
with $\mbox{free}(F)=\{x,x_1,\ldots,x_m\}$ 
and distinct variables $x,x_1,\ldots,x_m \in X$,
and we put
$$
\Phi(F) = \{F\frac{\lambda_0}{x}
\frac{\lambda_1}{x_1}\ldots\frac{\lambda_m}{x_m}\,:
\,\lambda_0,\lambda_1,\ldots,\lambda_m \in \hat{\mathcal{L}}_*\,\}\,.
$$
We write $\mbox{var}(\lambda) =\{y_1,\ldots,y_k\}$ and 
$\lambda = \lambda(y_1,\ldots,y_k)$\,.

From $x \in \mbox{free}(F)$ and $\mbox{CF}(F;\lambda;x)$ we see that 
$$\mbox{var}(\lambda)  \subseteq \mbox{free}\left(F\frac{\lambda}{x}\right)\,.$$
Hence we can write
$\mbox{free}(F\frac{\lambda}{x})=\{y_1,\ldots,y_n\}$
with $n \geq k$ distinct variables $y_1,\ldots,y_n\in X$
and define the new set
$$
\Phi(F;\lambda;x)=\{F\frac{\lambda}{x}
\frac{\mu_1}{y_1}\ldots\frac{\mu_n}{y_n}\,:
\,\mu_1,\ldots,\mu_n \in \hat{\mathcal{L}}_*\,\}\,.
$$
Again from $\mbox{CF}(F;\lambda;x)$ we conclude that
$$\qquad
\Phi(F;\lambda;x)=\left\{F\frac{\lambda(\mu_1,\ldots,\mu_k)}{x}
\frac{\mu_1}{y_1}\ldots\frac{\mu_n}{y_n}\,:
\,\mu_1,\ldots,\mu_n \in \hat{\mathcal{L}}_*\,\right\}\,,
$$
hence $\Phi(F;\lambda;x) \subseteq \Phi(F)$, and
we obtain from our induction hypothesis $\top =\mathcal{D}(\mbox{gen}(F))$
that $\top = \mathcal{D}(\mbox{gen}(F\frac{\lambda}{x}))$.\\
Now the extended proof $[\Lambda_*]=[\Lambda;F\frac{\lambda}{x}]$
satisfies our statement.

\textit{Rule (d).} Let $F$ be a step of the proof $[\Lambda]=[F_1;\ldots;F_l]$. 
Then $F$ is valid in $\mathcal{D}$ due to our induction hypothesis,
and Condition (7) gives 
$$\mathcal{D}(\mbox{gen}(F))=\mathcal{D}(\mbox{gen}(\forall x \,F))=\top$$
for any variable  $x \in X$, i.e. $\forall x \,F$ is also valid in $\mathcal{D}$.
Now the extended proof $[\Lambda_*]=[\Lambda;\forall x\,F]$
satisfies our statement.
\end{proof}
\begin{cor}\label{contra_free}
If the formal mathematical system $[M;\mathcal{L}]$
has a model $\mathcal{D}$, then it is consistent. 
\end{cor}
\begin{proof}
We assume the contrary and choose $x \in X$. Then
$\exists x \neg \sim\,x,x$ is provable in $[M;\mathcal{L}]$
and hence in its extension $[\hat{M};\hat{\mathcal{L}}]$.
We see from Lemma \ref{correct} that $\exists x \neg \sim\,x,x$ is 
valid in $\mathcal{D}$. Since we have $\hat{\mathcal{L}}_* \neq \emptyset$,
we have $\lambda \in \hat{\mathcal{L}}_*$ with
$\mathcal{D}(\neg \sim\,\lambda,\lambda)=\top$\,, i.e.
$\mathcal{D}(\sim\,\lambda,\lambda)=\bot$\,.
We obtain the contradiction $\mathcal{D}(\lambda) \neq \mathcal{D}(\lambda)$\,.
\end{proof}

For our new definition of a model we also obtain 
the following completeness theorem:

\begin{thm}\label{model_existence}
The formal mathematical system $[M;\mathcal{L}]$ is consistent
if and only if it has a model. 
\end{thm}
Up to now the proof of this theorem is only available in handwritten notes.
Based on the algebraic properties of the substitution, it makes use of Zorn's lemma and the construction of certain Henkin-extensions. 
In this paper we will not make use of Theorem \ref{model_existence} 
and hence omit its proof. But we also need Henkin-extensions of $[M;\mathcal{L}]$ and will introduce these in Section \ref{LS}.\\

\section{The downward L\"owenheim-Skolem Theorem}\label{LS}
Throughout the whole section let $[M;\mathcal{L}]$ with $M=[A;P;B]$ 
be a \textit{countable} formal mathematical system,
i.e. we assume that the sets $A$ and $P$ of symbols are countable.
Here we say that a set is countable if its either finite or countably infinite.

In the sequel we will assume  that $\mathcal{D}$ is a given
model for $[M;\mathcal{L}]$. Using the notation from Condition (3) for this model
we put $A''=\hat{A}$, 
$\begin{displaystyle}
B'' = \{ G \,:\, G \mbox{~is~a~formula~in~} [\hat{M};\hat{\mathcal{L}}]
\mbox{~which~is~valid~in~} \mathcal{D}\,\}
\end{displaystyle}$.

\begin{thm}\label{mzweistrich} The formal system $[M'';\mathcal{L}'']$
with $M''=[A'';P;B'']$ and $\mathcal{L}''=\hat{\mathcal{L}}$
is a complete Henkin system, i.e. we have
\begin{itemize}
\item[(a)] For $x \in X$ and every formula $F$ in 
$[M'';\mathcal{L}'']$ with 
$\mbox{free}(F) \subseteq \{x\}$ we have a constant symbol 
$\kappa \in \mathcal{L}''_*=\hat{\mathcal{L}}_*$ such that
\begin{equation*}
\leftrightarrow \,\exists x F~F\frac{\kappa}{x}~\in \Pi(M'';\mathcal{L}'')\,.
\end{equation*}
\item[(b)] For every formula $G$ in $[M'';\mathcal{L}'']$
with $\mbox{free}(G)=\emptyset$ we have
\begin{equation*}
G~\notin \Pi(M'';\mathcal{L}'') \Leftrightarrow
\neg G \in \Pi(M'';\mathcal{L}'')\,.
\end{equation*}
\end{itemize}
\end{thm}
\begin{proof}
(a) We know that $\exists x F$, $F\frac{\kappa}{x}$ and 
$\to\,\exists x F~F\frac{\kappa}{x}$ are closed formulas 
in $[\hat{M};\hat{\mathcal{L}}]$ for each constant $\kappa \in \hat{\mathcal{L}}_*$.
If $\mathcal{D}(\exists x F)=\top$, then we have $\lambda \in \hat{\mathcal{L}}_*$
with $\mathcal{D}(F \frac{\lambda}{x})=\top$\,.
We see from Lemma \ref{structure_lemma}(b) with the constant symbol
$\kappa = \alpha_{\mathcal{D}(\lambda)}\in \mathcal{N} \subseteq \hat{\mathcal{L}}_*$ that
\begin{equation*}
\mathcal{D}\left(F \frac{\lambda}{x}\right)=
\mathcal{D}\left(F \frac{\kappa}{x} \right)=\top \,.
\end{equation*}
We obtain 
$\begin{displaystyle} 
\mathcal{D}\left(\to\,\exists x F~F\frac{\kappa}{x}\right)=\top
\end{displaystyle}$
in this case. Otherwise we have 
$\mathcal{D}\left(\exists x F \right)=\bot$\,.
From $\mathcal{D}_* \neq \emptyset$ we can find a name $\kappa \in \mathcal{N}$
and see again $\mathcal{D}\left(\to\,\exists x F~F\frac{\kappa}{x}\right)=\top$.
We obtain from the definition of $B''$ that
$\begin{displaystyle}
\to\,\exists x F~F\frac{\kappa}{x}~\in \Pi(M'';\mathcal{L}'')
\end{displaystyle}$\,.
From \cite[(3.19) Proposition]{Ku} we also have
$\begin{displaystyle}
\to \, F\frac{\kappa}{x}\, \exists x F~\in \Pi(M'';\mathcal{L}'')
\end{displaystyle}$\,.\\
(b) Let $G$ be a formula in $[\hat{M};\hat{\mathcal{L}}]$ with
$\mbox{free}(G)=\emptyset$. For $\mathcal{D}(G)=\top$ we have
$G \in \Pi(M'';\mathcal{L}'')$ from the definition of $B''$.
If $\mathcal{D}(G)=\bot$, then $\mathcal{D}(\neg G)=\top$ and
$\neg G \in \Pi(M'';\mathcal{L}'')$. Hence
$G \in \Pi(M'';\mathcal{L}'')$ or $\neg G \in \Pi(M'';\mathcal{L}'')$.
Next we have to show that $[M'';\mathcal{L}'']$ is consistent.
If not then $\neg \forall x \sim x,x$ is provable in $[M'';\mathcal{L}'']$ 
for some $x \in X$. We put $\mbox{gen}(F)=\forall x_1 \ldots \forall x_j F$
for every formula $F$ in $[M'';\mathcal{L}'']$ 
with $\mbox{free}(F)=\{x_1,\ldots,x_j\}$ and distinct variables 
$x_1,\ldots,x_j \in X$, ordered according to their first occurrence in $F$. 
From the Deduction Theorem \cite[(4.3)]{Ku} we obtain
finitely many formulas $F_1, \ldots, F_k \in B''$ such that
\begin{equation*}
\to \mbox{gen}(F_1) \ldots \to \mbox{gen}(F_k) ~ \neg \forall x \sim x,x
~ \in \, \Pi(\hat{M};\hat{\mathcal{L}})\,.
\end{equation*}
Note that the closed formulas $\mbox{gen}(F_1), \ldots ,\mbox{gen}(F_k)$
are valid in $\mathcal{D}$. We obtain
$\begin{displaystyle}
\mathcal{D}(\mbox{gen}(F_1))= \ldots = \mathcal{D}(\mbox{gen}(F_k))
 =\top\,.
\end{displaystyle}$
It follows from Lemma \ref{correct} that 
$\begin{displaystyle}
\mathcal{D}\left(\to \mbox{gen}(F_1) \ldots \to \mbox{gen}(F_k) ~ \neg \forall x \sim x,x \right)
~ = \top
\end{displaystyle}$\,,
and we have the contradiction $\mathcal{D}\left(\neg \forall x \sim x,x \right)
~ = \top$, i.e.\\ 
$\mathcal{D}\left(\forall x \sim x,x \right)~ = \bot$.
\end{proof}
Now we extend $[M;\mathcal{L}]$ in a different way, starting from the following construction of \textit{Henkin-constants}:

\textit{Initial step:} Recall that $M=[A;P;B]$. Put $A^{(0)}=A$,
$\mathcal{L}^{(0)}=\mathcal{L}$, $M^{(0)}=M$\,. 
Let $x \in X$ and assume that $F$ is a formula in $[M^{(0)};\mathcal{L}^{(0)}]$ 
with $\mbox{free}(F) \subseteq \{x\}$. To this pair $(x;F)$ we choose exactly
one new constant $\omega_{\exists x F}$.
We call it a \textit{Henkin constant of rank 1}.\\
We put
$\begin{displaystyle}
\mathcal{H}^{(1)} = \{ c \,:\, c \mbox{~is~a~Henkin~constant~of~rank~1}\,\} 
\end{displaystyle}$\,,\\
$A^{(1)}=A^{(0)} \cup \mathcal{H}^{(1)}=A \cup \mathcal{H}^{(1)}$\,,
$M^{(1)}=[A^{(1)};P;B]$\,,
\begin{align*} {\mathcal L}^{(1)}   = 
\{\, & \lambda \frac{c_1}{x_1}...\frac{c_m}{x_m} \,\,:\,\, 
\lambda \in {\mathcal L}^{(0)},\,
x_1,...,x_m \in X,\,\\ & c_1,...,c_m \in \mathcal{H}^{(1)}
\,,\, m \geq 0\}\,.
\end{align*}

\textit{Inductive step:} Let $n \geq 2$ be a natural number.
Assume that for all $j \in \{1,\ldots,n-1\}$ we have already defined the formal mathe\-ma\-tical sys\-tems $[M^{(j)};\mathcal{L}^{(j)}]$ 
and the set $\mathcal{H}^{(j)}$
of Henkin constants of rank $j$. Let $x \in X$ and $\exists x F$ be a closed formula in $[M^{(n-1)};\mathcal{L}^{(n-1)}]$ such that at least one Henkin constant
of rank $n-1$ occurs in $F$. Then we form the new 
Henkin constant $\omega_{\exists x F}$ of rank $n$. 
By $\mathcal{H}^{(n)}$ we denote the set of all these Henkin constants of rank $n$. 
We use the notation $M^{(j)}=[A^{(j)};P;B]$ for $j=1,\ldots,n-1$ and put
$A^{(n)}=A^{(n-1)} \cup \mathcal{H}^{(n)}$, assuming that 
$A^{(n-1)} \cap \mathcal{H}^{(n)}=\emptyset$, $M^{(n)}=[A^{(n)};P;B]$ and
\begin{align*} {\mathcal L}^{(n)}   = 
\{\, & \lambda \frac{c_1}{x_1}...\frac{c_m}{x_m} \,\,:\,\, 
\lambda \in {\mathcal L}^{(n-1)},\,
x_1,...,x_m \in X,\,\\ & c_1,...,c_m \in \mathcal{H}^{(n)}
\,,\, m \geq 0\}\,.
\end{align*}
We define $\begin{displaystyle}
A^{(\infty)}=\bigcup \limits_{n=0}^{\infty} A^{(n)}
\end{displaystyle}$, $\begin{displaystyle}
\mathcal{L}^{(\infty)}=\bigcup \limits_{n=0}^{\infty} \mathcal{L}^{(n)}
\end{displaystyle}$ with the increasing chains
\begin{align}\label{henkin_ketten}
A = A^{(0)} \subset A^{(1)} \subset A^{(2)} \subset \ldots ~\,;~
{\mathcal L} = {\mathcal L}^{(0)} 
\subset {\mathcal L}^{(1)} 
\subset {\mathcal L}^{(2)} 
\subset \ldots ~\,.
\end{align}
With $M^{(\infty)}=[A^{(\infty)};P;B]$ we obtain a formal mathematical system
$[M^{(\infty)};\mathcal{L}^{(\infty)}]$ which results from $[M;\mathcal{L}]$
by adding all Henkin constants to $[M;\mathcal{L}]$.

Recall that $[M;\mathcal{L}]$ is countable. 
Then $[M^{(\infty)};\mathcal{L}^{(\infty)}]$ is countable too.
We summarize and keep in mind:

For $x \in X$ and every formula $F$ 
of $[M^{(\infty)};\mathcal{L}^{(\infty)}]$ 
with $\mbox{free}(F) \subseteq \{x\}$ we obtain a one-to-one correspondence between the pair $(x;F)$ and its Henkin constant $\omega_{\exists x F}$.

The set of all Henkin constants is given by 
$\begin{displaystyle}
\mathcal{H}=\bigcup \limits_{k=1}^{\infty} \mathcal{H}^{(k)} \,.
\end{displaystyle}$
Here we define a certain function $\psi : \mathcal{H} \to \mathcal{N}'$ with
a  \textit{countable subset} $\mathcal{N}' \subseteq \mathcal{N}$ of the set 
$\mathcal{N}$ of all names corresponding to the universe $\mathcal{D}_*$.
Recall that a set is countable if it is either finite or countably infinite.
We also extend the application of the function $\psi$ to all formulas of the countable formal mathematical system $[M^{(\infty)};\mathcal{L}^{(\infty)}]$.
Corresponding with the inductive construction of the Henkin extensions
$[M^{(n)};\mathcal{L}^{(n)}]$ for $n \in \setN_0$
we define a sequence $([\tilde{M}^{(n)};\tilde{\mathcal{L}}^{(n)}])_{n \in \setN_0}$
of countable subsystems of $[M'';\mathcal{L}'']$:

\textit{Initial step:} Let $\tilde{A}^{(0)}=A$ and note that 
$\begin{displaystyle}
\mathcal{L}_* = \{\,\lambda \in \mathcal{L}\,:\,\mbox{var}(\lambda)=\emptyset\,\}
\end{displaystyle}$
and $A$ are countable. We put
\begin{align*}
& \tilde{M}^{(0)}=[\tilde{A}^{(0)};P;\tilde{B}^{(0)}] \quad \mbox{with}\\
& \tilde{B}^{(0)} = \{ G \,:\, G \mbox{~is~a~formula~of~} [M;\mathcal{L}]
\mbox{~which~is~valid~in~} \mathcal{D}\,\} \quad \mbox{and}\\
&\tilde{\mathcal{L}}^{(0)}=\mathcal{L}\,.
\end{align*}
We define $\psi(F)=F$ for every formula $F$ 
of $[M^{(0)};\mathcal{L}^{(0)}]=[M;\mathcal{L}]$
from the initial step of the Henkin extensions.
Note that $[M'';\mathcal{L}'']$ in Theorem \ref{mzweistrich}
is an extension of $[\tilde{M}^{(0)};\tilde{\mathcal{L}}^{(0)}]$.

\textit{Inductive step:} 
\begin{itemize}
\item Suppose that for a given natural number $n \in \setN$ and for every 
$j \in \{0,\ldots,n-1\}$ we have already given the formal mathematical systems
$[\tilde{M}^{(j)};\tilde{\mathcal{L}}^{(j)}]$ 
with $\tilde{M}^{(j)}=[\tilde{A}^{(j)};P;\tilde{B}^{(j)}]$ and that
$[M'';\mathcal{L}'']$ in Theorem \ref{mzweistrich}
is an extension of $[\tilde{M}^{(j)};\tilde{\mathcal{L}}^{(j)}]$
for all $j \in \{0,\ldots,n-1\}$.
\item Suppose that the formula $\psi(F)$ of
$[\tilde{M}^{(n-1)};\tilde{\mathcal{L}}^{(n-1)}]$
is already defined for every formula $F$ of $[M^{(n-1)};\mathcal{L}^{(n-1)}]$.
\item Suppose that $\psi(c)$ is already defined for all Henkin constants $c$
of rank $<n$. \end{itemize}
Now let $\omega_{\exists x F}$ be a Henkin constant of rank $n$. 
Due to Theorem \ref{mzweistrich} and Lemma \ref{structure_lemma}
we can choose for each such formula $\exists x F$ exactly one individual 
$d \in \mathcal{D}_*$ with
\begin{align*}
\leftrightarrow \exists x \psi(F) \psi(F)\frac{\alpha_d}{x} 
\in \Pi(M'';\mathcal{L}'')\,,
\end{align*}
and we put $\psi\left(\omega_{\exists x F}\right)=\alpha_d$.
Now $\psi(c) \in \mathcal{N}$ is defined 
for all Henkin constants $c$ of rank $\leq n$.
Next we have to extend the definition of $\psi(F)$ to every formula $F$ of
$[M^{(n)};\mathcal{L}^{(n)}]$ which contains at least one Henkin constant 
of rank $n$ as a sublist: Let $c_1,\ldots,c_j$ with $j \geq 1$ be the complete list
of distinct Henkin constants of rank $n$ in $F$, ordered according to their first occurrence in $F$, and let $x_1,\ldots,x_j \in X$ be new distinct variables with
$x_1,\ldots,x_j \notin \mbox{var}(F)$\,.
Let the formula $G$ of $[M^{(n-1)};\mathcal{L}^{(n-1)}]$ result from $F$
if we replace $c_1,\ldots,c_j$ everywhere in $F$ by $x_1,\ldots,x_j$, respectively.
Then we put
\begin{align*}
\psi(F)=\psi \left( G \frac{c_1}{x_1} \ldots \frac{c_j}{x_j}\right)
=\psi\left( G \right) \frac{\psi(c_1)}{x_1} \ldots \frac{\psi(c_j)}{x_j}\,.
\end{align*}
Let $\begin{displaystyle}
\tilde{A}^{(n)}=\{\,\psi(c)\,:\,
c \mbox{~is~a~Henkin~constant~of~rank~} n\,\} 
\cup \tilde{A}^{(n-1)}
\end{displaystyle}$
and extend $[\tilde{M}^{(n-1)};\tilde{\mathcal{L}}^{(n-1)}]$ 
to the new formal mathematical system
$[\tilde{M}^{(n)};\tilde{\mathcal{L}}^{(n)}]$ with
$\begin{displaystyle}
\tilde{M}^{(n)}=[\tilde{A}^{(n)};P;\tilde{B}^{(n)}],
\end{displaystyle}$ and
\begin{align*}
& \tilde{B}^{(n)}=\{\,\psi\left( G \right)\,:\,
G \mbox{~is~a~formula~of~} [M^{(n)};\mathcal{L}^{(n)}]
\mbox{~and~} \psi(G) \mbox{~is~valid~in~} \mathcal{D}\,\} \,,\\
& \tilde{{\mathcal L}}^{(n)}   = 
\{\,\lambda \frac{\kappa_1}{x_1}...\frac{\kappa_m}{x_m} \,\,:\,\, 
\lambda \in \tilde{{\mathcal L}}^{(n-1)},\,
x_1,...,x_m \in X,\,\\ & \quad \quad \quad \quad
\kappa_1,...,\kappa_m \in \tilde{A}^{(n)} \setminus A
\,,\, m \geq 0\}\,.
\end{align*}
Recall that $\mathcal{N} \cap A = \emptyset$. Now $\psi(c) \in \mathcal{N}$
is defined for all Henkin constants $c$ of rank $\leq n$, and we see that
the formal system $[\tilde{M}^{(n)};\tilde{\mathcal{L}}^{(n)}]$ is an extension
of the formal systems $[\tilde{M}^{(j)};\tilde{\mathcal{L}}^{(j)}]$ 
with index $j<n$. Finally, the formal system $[M'';\mathcal{L}'']$
is an extension of $[\tilde{M}^{(n)};\tilde{\mathcal{L}}^{(n)}]$.

This concludes the inductive definition of $\psi(c) \in \mathcal{N}$
for all Henkin constants $c$ as well as the definition of
the formal systems $[\tilde{M}^{(n)};\tilde{\mathcal{L}}^{(n)}]$
for all $n \in \setN_0$. For all $n \in \setN$ 
and for every formula $F$ of $[M^{(n-1)};\mathcal{L}^{(n-1)}]$
we obtain a well defined formula $\psi(F)$ of
$[\tilde{M}^{(n-1)};\tilde{\mathcal{L}}^{(n-1)}]$,
and if in addition $\mbox{free}(F) \subseteq \{x\}$, then 
\begin{equation}\label{psi}
\left\{
\begin{split}
& \psi\left(\leftrightarrow \exists x F \, F\frac{c}{x}\right) ~=~\,
\leftrightarrow \exists x \, \psi(F) \, \psi(F)\frac{\psi(c)}{x} \mbox{~with~}\\
& 
\psi\left(\leftrightarrow \exists x F \, F\frac{c}{x}\right) \in \tilde{B}^{(n)}
\,,\mbox{~if~} c=\omega_{\exists x F} \mbox{~has~rank~} \leq n\,.
\end{split}
\right.
\end{equation}
Now we can define the following countable sets:
\begin{equation}\label{down_ketten}
\left\{
\begin{split}
A'&=\bigcup \limits_{n=0}^{\infty} \tilde{A}^{(n)}\,,\quad
\mathcal{N}'=A' \setminus A \subseteq \mathcal{N}\,,\\
B'&=\bigcup \limits_{n=0}^{\infty} \tilde{B}^{(n)}\,,\quad
\mathcal{L}'=\bigcup \limits_{n=0}^{\infty} \tilde{\mathcal{L}}^{(n)}\,.
\end{split}
\right.
\end{equation}
There results a formal mathematical system $[M';\mathcal{L}']$
with $M'=[A';P;B']$ which extends the original system $[M;\mathcal{L}]$.
It follows from our inductive construction 
that $\psi : \mathcal{H} \to \mathcal{N}'$ is surjective, and its extension
assigns to each formula $F$ of $[M^{(\infty)};\mathcal{L}^{(\infty)}]$
a well-defined formula $\psi(F)$ of $[M';\mathcal{L}']$.

\begin{thm}\label{DLST} {\bf Downward L\"owenheim-Skolem Theorem}
\begin{itemize}
\item[(a)] $[M';\mathcal{L}']$ is a complete countable Henkin system, and it extends the original countable system $[M;\mathcal{L}]$.
\item[(b)] The complete Henkin system $[M'';\mathcal{L}'']$
is a conservative extension of $[M';\mathcal{L}']$.
\item[(c)] For the model $\mathcal{D}$ of $[M;\mathcal{L}]$
we use the notations and Conditions (1)-(8) listed in Section \ref{models}.
\begin{itemize}
\item[1.] We form the countable subset 
$$\mathcal{D}_{0,*}=\{\,d\in \mathcal{D}_*\,:\,
\alpha_d \in \mathcal{N}'\,\} \subseteq \mathcal{D}_*$$ 
and note that $\mathcal{L}'_* \subseteq \mathcal{L}''_* 
= \hat{\mathcal{L}}_*$.
\item[2.] We define the restriction 
$\mathcal{D}_{0} : \mathcal{L}'_* \to \mathcal{D}_{0,*}$
of $\mathcal{D}$ to $\mathcal{L}'_*$ by 
$\mathcal{D}_{0}(\lambda)=\mathcal{D}(\lambda)$
for all $\lambda \in \mathcal{L}'_*$.
\item[3.] We put $p_0'=p_0$ and $p_n'=p_n \cap \mathcal{D}_{0,*}^{n}$ 
for all $n \geq 1$ with the predicates $p_0$ and $p_n$ from
Condition (5) in Section \ref{models}.
We use the predicates $p'_n$ for all $n \geq 0$ to extend $\mathcal{D}_{0}$
to all closed formulas $F$ of $[M';\mathcal{L}']$ by 
$\mathcal{D}_{0}(F)=\mathcal{D}(F)$.
\end{itemize}
Then $\mathcal{D}_{0}$ is a well defined countable model of $[M;\mathcal{L}]$,
and the models $\mathcal{D}$ and $\mathcal{D}_{0}$ are elementarily equivalent 
for $[M;\mathcal{L}]$. Here the names of $\mathcal{D}_{0}$ are the members
of the countable set ${\mathcal N}'$, where $\alpha_d \in {\mathcal N}'$
denotes the same individual $d \in \mathcal{D}_{0,*}$ in both models.
\end{itemize}
\end{thm}
\begin{proof}
(a) We have $A = \tilde{A}^{(0)}$, $B \subseteq \tilde{B}^{(0)}$, 
$\mathcal{L} = \tilde{\mathcal{L}}^{(0)}$ and hence 
$$
\Pi(M;\mathcal{L})
\subseteq \Pi(\tilde{M}^{(0)};\tilde{\mathcal{L}}^{(0)})\,.
$$
It follows from the inductive definition of $[M';\mathcal{L}']$
that 
$$\tilde{A}^{(n)} \subseteq \tilde{A}^{(n+1)} \subseteq A'\,, \quad 
\tilde{B}^{(n)} \subseteq \tilde{B}^{(n+1)} \subseteq B'\,, \quad
\tilde{\mathcal{L}}^{(n)} \subseteq \tilde{\mathcal{L}}^{(n+1)} 
\subseteq \mathcal{L}'$$
and hence 
$
\Pi(\tilde{M}^{(n)};\tilde{\mathcal{L}}^{(n)})
\subseteq \Pi(\tilde{M}^{(n+1)};\tilde{\mathcal{L}}^{(n+1)})
$
for all $n \in \setN_0$. We see that $[M';\mathcal{L}']$
is a countable extension of $[M;\mathcal{L}]$.
Due to the definition of $B'$ we have 
$F' \in B' \Rightarrow F' \in \Pi(M'';\mathcal{L}'')$.
We see that $[M'';\mathcal{L}'']$ is a consistent extension of 
$[M';\mathcal{L}']$, see Theorem \ref{mzweistrich}(b).
Hence $[M';\mathcal{L}']$ is consistent too.

Let $F'$ be a formula of $[M';\mathcal{L}']$ 
with $\mbox{free}(F') \subseteq \{x\}$, $x \in X$.
For the complete list of names $\kappa_1,\ldots,\kappa_j \in \mathcal{N}'$
occurring in $F'$ we can choose
Henkin constants $c_1,\ldots,c_j$ (including $j=0$) with
$\psi(c_l)=\kappa_l$, $l=1,\ldots,j$, and we obtain a formula $F$ from $F'$
if we replace $\kappa_1,\ldots,\kappa_j$ everywhere in $F'$ by $c_1,\ldots,c_j$,
respectively. For $j=0$ we put $n=0$, for $j \geq 1$ let $n$ be the maximal rank
of the Henkin constants $c_1,\ldots,c_j$. Then $F$ is a formula
of $[M^{(n)};\mathcal{L}^{(n)}]$ with  $F'=\psi(F)$ and
$\mbox{free}(F) \subseteq \{x\}$. It follows from \eqref{psi} that
\begin{align*}
\leftrightarrow \exists x \, F' \, F' \frac{\psi(\omega_{\exists x F})}{x}
\in B'
\end{align*}
with $\psi(\omega_{\exists x F}) \in \mathcal{N}'$, and 
$[M';\mathcal{L}']$ is a consistent Henkin system.

Let $G'$ be a formula of $[M';\mathcal{L}']$ with $\mbox{free}(G')=\emptyset$.
Then we can find $n \in \setN_0$ and a formula $G$ in $[M^{(n)};\mathcal{L}^{(n)}]$ with  $G'=\psi(G)$ and $\mbox{free}(G)=\emptyset$. Recall that $G'$ is also
a formula of $[M'';\mathcal{L}'']$ as well as in $[\hat{M}; \hat{\mathcal{L}}]$.
If $\mathcal{D}(G')=\top$, then $G' \in \tilde{B}^{(n)} \subseteq B'$,
hence $G' \in \Pi(M';\mathcal{L}')$.
Otherwise $\mathcal{D}(\neg G')=\top$ and
$\neg G' \in \tilde{B}^{(n)} \subseteq B'$, hence
$\neg G' \in \Pi(M';\mathcal{L}')$. We see that
$[M';\mathcal{L}']$ is a complete Henkin system.

(b) Let $G'$ be a formula of $[M';\mathcal{L}']$ which is provable in 
$[M'';\mathcal{L}'']$ and let $G$ be a formula of
$[M^{(\infty)};\mathcal{L}^{(\infty)}]$ with $G'=\psi(G)$. 
Then we have variables $x_1,\ldots, x_j \in X$ with
$$ \mbox{free}(G')=\mbox{free}(G)= \{x_1,\ldots, x_j\}\,.$$
Assume that $G' \notin \Pi(M';\mathcal{L}')$.
Then $\forall x_1 \ldots \forall x_j G' \in \Pi(M'';\mathcal{L}'')$,
but $\forall x_1 \ldots \forall x_j G' \notin \Pi(M';\mathcal{L}')$.
We obtain $\mathcal{D}(\forall x_1 \ldots \forall x_j G')=\top$
from the definition of $B''$, and hence
$$
\mathcal{D}(\psi(\forall x_1 \ldots \forall x_j G))=
\mathcal{D}(\forall x_1 \ldots \forall x_j G')=\top\,,
$$
$\forall x_1 \ldots \forall x_j G'\in B'$
from the definition of $B'$, which contradicts our assumption.
We have shown (b).

(c) First we have to show that the restriction $\mathcal{D}|_{\mathcal{L}'_*}$
of the function $\mathcal{D} : \mathcal{L}''_* \to \mathcal{D}_*$ 
to $\mathcal{L}'_*$ is a surjective function 
$\mathcal{D}_0 : \mathcal{L}'_* \to \mathcal{D}_{0,*}$ \,,
i.e. that the range of $\mathcal{D}|_{\mathcal{L}'_*}$ is $\mathcal{D}_{0,*}$:
We prescribe $\mu' \in \mathcal{L}'_*$, choose $x \in X$ and put 
$G'=\exists x \sim \mu',x$. Then $G'$ is a closed formula of $[M';\mathcal{L}']$,
and we can find a corresponding closed formula $G = \exists x \sim \mu,x$
of $[M^{(\infty)};\mathcal{L}^{(\infty)}]$ with $\mu \in \mathcal{L}^{(\infty)}_*$
and $G'=\psi(G)$.
We put $c = \omega_{\exists x \sim \mu,x}$ and have 
$\psi(c) = \alpha_d \in \mathcal{N}'$ for some $d \in \mathcal{D}_{0,*}$,
which is uniquely determined from $c$ via the surjective function
$\psi : \mathcal{H} \to \mathcal{N}'$. We have defined the function $\psi$
in such a way that we obtain with $F = \sim \mu,x$ from \eqref{psi}:
$$
\mathcal{D}\left(\leftrightarrow 
\exists x \sim \mu',x
\sim \mu',\psi(c)
\right)= \top\,.
$$
We use 
$\begin{displaystyle}
\mathcal{D}\left(
\exists x \sim \mu',x
\right)= \top 
\end{displaystyle}$
and see that
$$
\mathcal{D}\left(\sim \mu',\psi(c)\right)= 
\mathcal{D}\left(\sim \mu',\alpha_d \right)=\top \,,
$$
i.e. 
$\begin{displaystyle}
\mathcal{D}_0\left(\mu'\right)= \mathcal{D}\left(\mu'\right)= 
\mathcal{D}\left(\alpha_d\right)= d \in \mathcal{D}_{0,*}\,.
\end{displaystyle}$
On the other hand we have
$\begin{displaystyle}
\mathcal{D}_0\left(\alpha_d \right)= \mathcal{D}\left(\alpha_d\right)= d 
\end{displaystyle}$
for any given $d \in \mathcal{D}_{0,*}$. 
Hence the range of $\mathcal{D}|_{\mathcal{L}_*'}$ is given by $\mathcal{D}_{0,*}$.
Now $\mathcal{D}_0 : \mathcal{L}'_* \to \mathcal{D}_{0,*}$
with $\begin{displaystyle}
\mathcal{D}_0\left(\mu'\right)= \mathcal{D}\left(\mu'\right)
\end{displaystyle}$
is a well-defined surjective function, and we have
\begin{equation*}
\begin{split}
& \mathcal{D}_0\left( \lambda' \frac{\mu'}{x}\right)=
\mathcal{D}\left( \lambda' \frac{\mu'}{x}\right)\\
=~& \mathcal{D}\left( \lambda' \frac{\alpha_{\mathcal{D}(\mu')}}{x}\right)=
\mathcal{D}_0\left( \lambda' \frac{\alpha_{\mathcal{D}(\mu')}}{x}\right)\\
\end{split}
\end{equation*}
for all $\lambda' \in \mathcal{L}'$ with 
$\mbox{var}(\lambda') \subseteq \{x\}$
and for all $\mu' \in \mathcal{L}'_*$
from the corres\-pon\-ding property of the model $\mathcal{D}$.
We have added the constant symbols in the set $\mathcal{N}$ of names
to $[M;\mathcal{L}]$ and obtained the formal mathematical system 
$$
[\hat{M};\hat{\mathcal{L}}]=[\,[A'';P;B];\mathcal{L}'']\,.
$$
In the same way we obtain 
the formal mathematical subsystem 
$$
[\,[A';P;B];\mathcal{L}']\,,
$$
if we only add the names from the subset $\mathcal{N}' \subseteq \mathcal{N}$
to $[M;\mathcal{L}]$. This subsystem plays the same role
for the desired submodel $\mathcal{D}_0$ as $[\hat{M};\hat{\mathcal{L}}]$
plays for the model $\mathcal{D}$ of $[M;\mathcal{L}]$\,,
see Condition (3) in Section \ref{models}.
Recall that
$\begin{displaystyle}
\mathcal{D}_0\left(\alpha_d \right)= \mathcal{D}\left(\alpha_d\right)= d 
\end{displaystyle}$
for any given $d \in \mathcal{D}_{0,*}$.
Hence we use the same names for the same individuals 
in our subsystem, and we have proved that $\mathcal{D}_0$
satisfies Conditions (1)-(4) in Section \ref{models}
required for a model of $[M;\mathcal{L}]$. 
We are checking the remaining Conditions (5)-(8):

(5) We have already defined $p_0' = p_0$ 
and $p_n' = p_n \cap \mathcal{D}_{0,*}^n$ for all $n \in \setN$
and use them first for the interpretation of prime formulas of $[M';\mathcal{L}']$.
Note that the formulas of $[M';\mathcal{L}']$ and 
$
[\,[A';P;B];\mathcal{L}']
$
are the same.

(6) We have \textit{defined} $\mathcal{D}_0(F')=\mathcal{D}(F')$
for every closed formula $F'$ of $[M';\mathcal{L}']$.
Since $\mathcal{D}_0(\lambda')=\mathcal{D}(\lambda') \in \mathcal{D}_{0,*}$
for all $\lambda' \in \mathcal{L}_*'$, we see that for all 
$\lambda',\mu',\lambda_1',\ldots,\lambda_n' \in \mathcal{L}_*'$
with $n \in \setN$:
\begin{itemize}
\item[6.1] $\begin{displaystyle}
\mathcal{D}_0(\sim \lambda',\mu')
=\mathcal{D}(\sim \lambda',\mu')=\top \Leftrightarrow
\mathcal{D}_0(\lambda')=\mathcal{D}_0(\mu') \in \mathcal{D}_{0,*}\,.
\end{displaystyle}$ 
\item[6.2] For all $p \in P$ we have
$\mathcal{D}(p)=\mathcal{D}_0(p)=p_0 \in \{\top ,\bot\}$ 
and\\
$\begin{displaystyle}
\mathcal{D}_0(p \, \lambda'_1,\ldots,\lambda'_n)=\top \Leftrightarrow
\mathcal{D}(p \, \lambda'_1,\ldots,\lambda'_n)=\top \Leftrightarrow
\end{displaystyle}$ \\
$\begin{displaystyle}
(\mathcal{D}(\lambda'_1),\ldots,\mathcal{D}(\lambda'_n)) \in p_n \Leftrightarrow
\end{displaystyle}$ \\
$\begin{displaystyle}
(\mathcal{D}_0(\lambda'_1),\ldots,\mathcal{D}_0(\lambda'_n)) 
\in p_n \cap \mathcal{D}_{0,*}^n = p'_n\,.
\end{displaystyle}$ 
\item[6.3] Let $F',G'$ be closed formulas of $[M',\mathcal{L}']$. Then\\
$\begin{displaystyle}
\mathcal{D}_0(\neg F')=\top \Leftrightarrow
\mathcal{D}(\neg F')=\top \Leftrightarrow 
\mathcal{D}(F')=\bot \Leftrightarrow 
\mathcal{D}_0(F')=\bot \,,
\end{displaystyle}$\\
$\begin{displaystyle}
\mathcal{D}_0(\to F' G')=\top \Leftrightarrow
\mathcal{D}(\to F' G')=\top \Leftrightarrow
\end{displaystyle}$ \\
$\begin{displaystyle}
(\mathcal{D}(F') \Rightarrow \mathcal{D}(G')) \Leftrightarrow
(\mathcal{D}_0(F') \Rightarrow \mathcal{D}_0(G'))
\end{displaystyle}$\,, \\
and similarly for $\leftrightarrow$, $\&$ and $\vee$\,.
\item[6.4] For $x \in X$ and every formula 
$H'$ of $[M';\mathcal{L}']$ with $\mbox{free}(H') \subseteq \{x\}$
we can find a formula $H$ of $[M^{(\infty)};\mathcal{L}^{(\infty)}]$
with $H'=\psi(H)$, and we have again $\mbox{free}(H) \subseteq \{x\}$.
By $n \in \setN$ we denote the rank of $\omega_{\exists x H}$,
and we put $\kappa_+=\psi(\omega_{\exists x H}) \in \mathcal{N}'$.
It follows from \eqref{psi} with $F=H$ and $c=\omega_{\exists x H}$ that
\begin{align*}
\mathcal{D}_0\left(\leftrightarrow \, \exists x H' \, H'\frac{\kappa_+}{x}\right)=
\mathcal{D}\left(\leftrightarrow \, \exists x H' \, H'\frac{\kappa_+}{x}\right)=
\top \,,
\end{align*}
hence~
\begin{align*}
\mathcal{D}_0(\exists x \, H')=
\mathcal{D}(\exists x \, H')=
\mathcal{D}(H'\frac{\kappa_+}{x})=
\mathcal{D}_0(H'\frac{\kappa_+}{x})
\in \{\top,\bot\}\,,
\end{align*}
similarly for
$\kappa_-=\psi(\omega_{\exists x \neg H}) \in \mathcal{N}'$:
\begin{align*}
\mathcal{D}\left(\exists x \, \neg H'\right)=
\mathcal{D}\left(\neg H'\frac{\kappa_-}{x}\right)
\in \{\top,\bot\}\,,
\end{align*}
and finally
\begin{align*}
\mathcal{D}_0\left(\forall x \, H'\right)=
\mathcal{D}_0\left(H'\frac{\kappa_-}{x}\right)
\in \{\top,\bot\}\,.
\end{align*}
We conclude that it is equivalent for the evaluation of the closed formulas
$Q x H'$ of $[M';\mathcal{L}']$ with $Q \in \{\exists,\forall\}$ 
to interpret them by using the names either in $\mathcal{N}'$ 
or in $\mathcal{N}$\,, which is in accordance with part (b) of this theorem.
\end{itemize}
(7) Let $F'$ be a formula of $[M';\mathcal{L}']$ with 
$\mbox{free}(F') = \{x_1,\ldots,x_m\}$, $x_1,\ldots,x_m \in X$ 
and $m \geq 0$. We say that $F'$ is valid in $\mathcal{D}_0$ iff
$$
\mathcal{D}_0\left(
F' \frac{\lambda'_1}{x_1}\ldots\frac{\lambda'_m}{x_m}
\right)=\top
$$
for all $\lambda'_1,\ldots,\lambda'_m \in \mathcal{L}'_* $.
This is a definition for the formulation of Condition (8),
and it is equivalent with
$\begin{displaystyle}
\mathcal{D}_0\left(
\forall x_1 \ldots \forall x_m F' 
\right)=\top
\end{displaystyle}$\,.
But we note that $F' \in \Pi(M';\mathcal{L}')$ iff
$\forall x\, F' \in \Pi(M';\mathcal{L}')$
from \cite[(3.11)(a),(3.13)(b)(d)]{Ku}.
Therefore Condition (7) is in accordance with Lemma \ref{correct}.

(8) Let $F' \in B'$, $\mbox{free}(F')=\{x_1,\ldots,x_m\}$
with $m \geq 0$. Then we have $F' \in \tilde{B}^{(n)}$ for some $n \geq 0$,
and $F'$ is valid in $\mathcal{D}$. We obtain
$$\mathcal{D}_0\left(
\forall x_1 \ldots \forall x_m F' 
\right)=
\mathcal{D}\left(
\forall x_1 \ldots \forall x_m F' 
\right)=\top\,,
$$
and from $B \subseteq B'$ we see that $\mathcal{D}_0$ is a model
for $[M;\mathcal{L}]$. 

Finally let $F$ with $\mbox{free}(F)=\{x_1,\ldots,x_m\}$
with $m \geq 0$ be a formula of $[M;\mathcal{L}]$. 
Then $F$ is also a formula of $[M';\mathcal{L}']$,
and we obtain from the definition of $\mathcal{D}_0$ that
\begin{align*}
\mathcal{D}_0\left(
\forall x_1 \ldots \forall x_m F 
\right)=
\mathcal{D}\left(
\forall x_1 \ldots \forall x_m F 
\right) \in \{\top,\bot\}\,.
\end{align*}
Hence $\mathcal{D}$ and $\mathcal{D}_0$ 
are elementarily equivalent models for $[M;\mathcal{L}]$. 
\end{proof}

\begin{rem}

(a) The model $\mathcal{D}_0$ in Theorem \ref{DLST} is also called an 
\textit{elementary countable submodel} of $\mathcal{D}$.

(b) Note that $\mathcal{D}_*$ need not be infinite.
In contrast to the strict inclusions in \eqref{henkin_ketten}
the unions of the sets in \eqref{down_ketten} may terminate for some index $n \in \setN$.

\end{rem}

\section{Application to axiomatic set theory}\label{set_theory}

Axiomatic set theory provides some generally accepted rules for dealing with sets, and it intends to lay a foundation of mathematics by using only the primitive terms ``set" and ``membership".
More gene\-ral mathe\-matical structures are defined using these primitive terms.\\
Axio\-matic set theory should then prove as many relevant theorems about the general structures as possible. Such a commonly accepted foundation of mathe\-ma\-tics is the Zermelo-Fraenkel set theory ZFC with the axiom of choice,
see Jech \cite{J} and Shoenfield \cite[Chapter 9]{Shoenfield}.
ZFC is only dealing with sets whose members are sets again.
A more general approach to set theory also allows ``urelements''
as members of sets and it is studied by Moschovakis \cite{Mosch}.

\noindent
In \cite{Kuset} we have presented a generalization of ZFC, 
starting with a fragment of axiomatic set theory
called RST, for reduced set theory.
As in ZFC we are only dealing with sets whose members are sets again.

\noindent
A set $U$ is called \textit{transitive} iff $Y \subseteq U$ 
for all $Y \in U$.\\
By $\mathcal{P}(Y) =\{ V : V \subseteq Y\}$
we denote the power set of $Y$.\\
Due to \cite[Definition 2.1]{Kuset} 
we say a set $U$ is \textit{subset-friendly} iff 
\begin{itemize}
\item[1.] $\emptyset \in U$\,.
\item[2.] $U$ is transitive\,.
\item[3.] For all $Y \in U$ we have $\mathcal{P}(Y) \in U$\,.
\item[4.] For all $Y, Z \in U$ we have a transitive set\\ 
$V \in U$ with $\{Y,Z\} \subseteq V$\,.
\end{itemize}

\noindent
Now we are listing the six principles 
according to which we are dealing with sets in RST. 
For sets $A$, $B$, $U$, $V$, $Y$ these are given by
\begin{itemize}
\item[P1.] \textit{Principle of extensionality.}
If $A$ and $B$ have the same elements, then $A=B$.
\item[P2.] \textit{Subset principle.}
If $\mathcal{F}$ is a property which may depend on previously given sets, then we can form the subset of $A$ given by
$U = \{Y :\,\mbox{there~holds~} Y \in A 
\mbox{~and~} Y \mbox{~has~property~}\mathcal{F}\}\,.$\\
Especially the empty set $\emptyset$ can be obtained from this principle.
\item[P3.] \textit{Principle of regularity.}
If $U$ is not the empty set,
then we have $Y \in U$ with $U \cap Y=\emptyset$.
\item[P4.] \textit{Principle for pairing of sets.}
If $A$ and $B$ are given, then we can find a set $U$
with $A \in U$ and $B \in U$. 
We can combine this with (P2)
to form $U=\{A, B\}$.
\item[P5.] \textit{Principle for subset-friendly sets.}
If $A$ is given, then we have a subset-friendly set $U$
with $A \in U$.
\item[P6.] \textit{Principle of choice.}
If $U$ has only nonempty and pairwise disjoint elements
then we can find a set $Y$ with the following property:
For every member $A \in U$ there exists exactly one set $V$ with
$Y \cap A = \{V\}$.
\end{itemize}
The novel feature of (P5) is that it contains the set $A$ as 
para\-meter. Hence we can use it step by step.
We will first provide a subset-friendly set $U$ with $A=\emptyset \in U$.
Then we can apply (P5) to $A=U$ again, and so on.
The correctness of (P5) is guaranteed by \cite[Theorem 2.5]{Kuset}.
In this way we have a sufficiently large set as background available. 
Within this set we can perform the set operations
listed in \cite[Remark 4.11]{Kuset}. Then we apply the subset axioms directly instead of the replacement axioms.

To formulate the axioms (A1)-(A6) in \cite[Section 4]{Kuset}
corresponding to (P1)-(P6) we have used 
the formal language of the predicate calculus. 
The language of set theory consists of a set
$X=\{\,{\bf x_1}\,,\,{\bf x_2}\,,\,{\bf x_3}\,,\, \ldots \,\}$
of variables, the equality sign $\sim$ and a binary predicate symbol
$\in$ for membership relation. Using variables $x,y \in X$
we start with atomic formulas $\sim x,y$ and $\in x,y$.
Let formulas $F$, $G$ be constructed pre\-viously.
Then we can form step by step the connectives 
$$
\neg F\,, \quad \to F G\,, \quad \& F G\,, \quad \vee F G\,, \quad
\leftrightarrow F G
$$
and the formulas
$\begin{displaystyle}
\forall x\, F\,, ~ \exists x\,F
\end{displaystyle}$\,.
Note that \cite[Section 3]{Ku} does not provide a restriction to binary prime formulas $\in x,y$. But that doesn't matter because we
could use \cite[Section 3.4]{Ku2} to eliminate non-binary prime formulas 
with the symbol $\in$ from the formal proofs. Therefore, in \cite{Kuset} we use this restriction on the RST-formulas from the beginning.

In \cite[Theorem 5.1]{Kuset} we obtained a hierarchy of transitive 
models for RST. The universe of each model is a 
subset-friendly set $\mathcal{U}_n$, and the membership relation in each model is the true membership relation between the individuals 
in the universe $\mathcal{U}_n$. These are only the simplest transitive models. 
All of them have only finite or countably infinite ordinals. 
The models for RST from \cite[Theorem 5.1]{Kuset} without uncountable ordinals 
violate Zermelo's well-ordering theorem.
Zermelo's well-ordering theorem states that for every set $A$ 
there is a bijective mapping from a von Neumann ordinal to $A$.
\footnote{Nevertheless, any set can be well-ordered within the frame of RST; see (P6).}
We mention two reasons for studying RST.
The first reason is that RST admits transitive models
which can be extended by adding step by step appropriate new axioms to RST.
Then the former transitive model just becomes a transitive set
and a member of the extended model. In this way we can extend RST
and its transitive models, whenever this is needed.
On the other hand, we have seen in \cite{Kuset} that even the simplest
models of RST are large and rich enough in order to formalize
most parts of classical mathema\-tics. Hence we can also study
axiomatic set theory within the theory of models, 
using universal sets instead of proper classes.

As a second reason, we express the following reservation
about the replacement axioms:
Adding the replacement axioms to RST we obtain an axiomatic system
which is equivalent to ZFC. In this system we can construct the 
von Neumann hierarchy of sets as follows:

We use the notation $\alpha<\beta$
in order to indicate that $\alpha$ and $\beta$ are
(von Neumann) ordinals with $\alpha \in \beta$,
and we put $\alpha+1=\alpha \cup \{\alpha\}$,
$V_{\emptyset}=\emptyset$ as well as
$V_{\alpha+1}=\mathcal{P}(V_{\alpha})$,
$V_{\beta}=\bigcup_{\alpha<\beta} V_{\alpha}$
for all ordinals $\alpha$ and for all limit ordinals $\beta$.
Then for every given set $A$ we have 
an ordinal $\alpha$ with $A \in V_{\alpha}$.
Now
$
V_{\omega_1} 
$
with $\omega_1 = \{\alpha\,:\, \alpha \mbox{ is a countable ordinal}\,\}$
turned out to be a model for RST, and we obtained

\begin{thm}\label{finalthm2} \cite[Theorem 5.6]{Kuset}:
We have a countable transitive set $\mathcal{U} \in V_{\omega_1}$ such
that the true binary membership relation 
$$
E=\{(A,B)\,:\, A \in B \mbox{~and~}
 A \in \mathcal{U} \mbox{~and~}  B \in \mathcal{U} \,\}
$$
makes $\mathcal{U}$ a model for RST 
which is elementarily equivalent to $V_{\omega_1}$.
\end{thm}

Next we derive an important generalization of Theorem \ref{finalthm2},
and we use the same ingredients for its proof. Here we present
again \cite[Lemma 5.5]{Kuset} for the sake of completeness:
 
\begin{lem}\label{omegalem}
Let $(A_k)_{k \in \setN}$ be a sequence of sets 
with $A_k \in V_{\omega_1}$ for all $k \in \setN$.
Then we have
$\begin{displaystyle}
\{ A_k\,:\,k \in \setN\} \in V_{\omega_1}\,.
\end{displaystyle}$
\end{lem}
\begin{proof}
We put $\mathcal{A}=\{ A_k\,:\,k \in \setN\}$. 
It follows from the definition of $V_{\omega_1}$ that 
for all $k \in \setN$ we have a countable ordinal $\alpha_k$ 
with $A_k \in V_{\alpha_k}$. Hence we can form the 
countable ordinal 
$\begin{displaystyle}
\alpha = \bigcup_{k \in \setN} \alpha_k \in \omega_1
\end{displaystyle}$ and see that
$
\mathcal{A} \subseteq V_{\alpha} \in V_{\omega_1}\,.
$
Since $\mathcal{A} \in \mathcal{P}(V_{\alpha})=V_{\alpha+1} \in V_{\omega_1}$, we obtain the desired result from the transitivity of $V_{\omega_1}$.
\end{proof}

Theorem \ref{finalthm2} has the following natural extension:

\begin{thm}\label{finalthmext} 
Let $\beta$ be an ordinal 
and assume that $V_{\beta}$ is a model of RST.
Then we have a countable transitive set $\mathcal{U} \in V_{\omega_1}$ such
that the true binary membership relation 
$$
E=\{(A,B)\,:\, A \in B \mbox{~and~}
 A \in \mathcal{U} \mbox{~and~}  B \in \mathcal{U} \,\}
$$
makes $\mathcal{U}$ a model for RST 
which is elementarily equivalent to $V_{\beta}$.
\end{thm}
\begin{proof} The downward L\"owenheim-Skolem Theorem \ref{DLST}
gives an elementary equivalent countable submodel 
$\mathcal{U}' \subseteq V_{\beta}$ for RST.
Due to (P1) the set $\mathcal{U}'$ is extensional in the sense of 
\cite[Definition 6.14]{J}, i.e. for any two distinct sets
$A, B \in \mathcal{U}'$ we have 
$A \cap \mathcal{U}' \neq B \cap \mathcal{U}'$.
Mostowski's Collapsing Theorem \cite[Theorem 6.15(ii)]{J}
gives a transitive set $\mathcal{U}$ and
an isomorphism $f : \mathcal{U}' \to \mathcal{U}$
such that $A \in B \Leftrightarrow f(A) \in f(B)$ 
for all $A, B \in \mathcal{U}'$. We see that $f$ preserves elementary equivalence. 

It remains to prove that $\mathcal{U} \in V_{\omega_1}$.
Due to Lemma \ref{omegalem} it is sufficient to show that
$\mathcal{U} \subseteq V_{\omega_1}$. 
We define $\tilde{f} : V_{\beta} \to \mathcal{U}$ by
\begin{equation*}
	\tilde{f}(A) =\begin{cases}
	f(A) &\text{if}~ A \in \mathcal{U}',\\
	\emptyset &\text{otherwise}.
	\end{cases}
	\end{equation*}
We say that a set $A \in V_{\beta}$ has property
$\Phi(A)$ iff 
$
\tilde{f}(A) \in V_{\omega_1}\,.
$

Let $A \in V_{\beta}$ and assume $\Phi(B)$ for all $B \in A$.
If $A \notin \mathcal{U}'$ then $\tilde{f}(A)=\emptyset \in V_{\omega_1}$
and $\Phi(A)$. Otherwise we have $A \in \mathcal{U}'$ and
$
\tilde{f}(A) = f(A)=\{f(B) : B \in A \cap \mathcal{U}'\}\,.
$
For every $C \in \tilde{f}(A)$ we can find 
$B \in A \cap \mathcal{U}'$ with $C=f(B)=\tilde{f}(B)$, and we 
obtain $C \in V_{\omega_1}$ from $\Phi(B)$. 
We see that $\tilde{f}(A) \subseteq V_{\omega_1}$ and
have $\tilde{f}(A) \in V_{\omega_1}$ from Lemma \ref{omegalem} since
$\tilde{f}(A)$ is countable.
Hence there holds $\Phi(A)$ 
whenever there holds $\Phi(B)$ for every $B \in A$.
We also have $\tilde{f}(\emptyset)=\emptyset$ and
obtain $\Phi(\emptyset)$ from $\emptyset \in V_{\omega_1}$.
Finally we see by $\in$-induction from \cite[Theorem 6.4]{J} 
that $\Phi(A)$ is valid for all $A \in V_{\beta}$.
Since the image of $\tilde{f}$ and $f$ 
is the countable transitive set $\mathcal{U}$,
we conclude that $\mathcal{U} \subseteq V_{\omega_1}$.
\end{proof}

Let $\beta$ be an ordinal with a universal set $V_{\beta}$ for RST
and let $\mathcal{U} \in V_{\omega_1}$ be the corresponding 
countable transitive model for RST from Theorem \ref{finalthmext}. 
Then both models $\mathcal{U}$ and $V_{\beta}$ for RST 
satisfy exactly the same statements, 
expressed in the formal language of set theory.\\

We make use of the notation from ordinal arithmetic and consider two cases: 
In the first case let $\beta < \omega_1$ be a ``small" countable ordinal like
$\beta = \omega \cdot \omega$, such that
$\mathcal{U}$ contains the same ordinals like $V_{\beta}$. 
In this case we have $\mathcal{U} \notin V_{\beta}$.\\

In the second case let $\omega_1 < \beta$, 
for example with $\beta = \omega_{1} + (\omega \cdot \omega)$. 
But here we have $\mathcal{U} \in V_{\omega_1} \in V_{\beta}$
with three transitive models for RST,
where $\mathcal{U}$ and $V_{\beta}$
give elementary equivalent models. \\

Concerning the second case 
I would like to see a good reason 
why the use of uncountable ordinals like $\omega_1$ 
is not in conflict with the principle (P3) of regularity.
Therefore we proclaim a moderate parsimony principle 
in the choice of mathematical tools, 
instead of risky maximality properties 
for a single universe of sets.

\end{document}